\documentclass[american,10pt]{amsart}
\pdfoutput=1
\usepackage[T1]{fontenc}
\usepackage[utf8]{inputenc}
\usepackage{lmodern}
\usepackage{enumitem}
\usepackage{mathtools}
\usepackage{amssymb}
\usepackage{bm}
\usepackage[unicode=true,pdfusetitle,
 bookmarks=true,bookmarksnumbered=false,bookmarksopen=false,
 breaklinks=false,pdfborder={0 0 1},backref=false,colorlinks=false]{hyperref}
\usepackage[foot]{amsaddr}
\usepackage{cite}

\theoremstyle{plain}
\newtheorem{thm}{Theorem}[section]
\newtheorem{lem}[thm]{Lemma}
\newtheorem{prop}[thm]{Proposition}

\newtheorem{cor}[thm]{Corollary}

\theoremstyle{definition}

\theoremstyle{remark}
\newtheorem{rem}{Remark}[section]
\newtheorem*{rem*}{Remark}

\newcommand\R{\mathbb{R}}
\newcommand\Sn{\mathbb{S}^{n-1}}
\newcommand\N{\mathbb{N}}
\newcommand\Z{\mathbb{Z}}
\DeclareMathOperator{\vspan}{span}
\DeclareMathOperator{\diag}{diag}
\newcommand\upd{\textup{d}}
\newcommand\upT{\textup{T}}
\newcommand\upp{\textup{per}}
\newcommand\upe{\textup{e}}
\newcommand\upDir{\textup{Dir}}

\newcommand\upini{\textup{ini}}
\newcommand\upFG{\textup{FG}}

\newcommand{\vect}[1]{\bm{\mathbf{#1}}} 
\newcommand\veL{\vect{L}} 
\newcommand\veC{\vect{C}} 
\newcommand\veu{\vect{u}} 
\newcommand\vev{\vect{v}} 
\newcommand\vew{\vect{w}} 
\newcommand\vez{\vect{0}} 
\newcommand\veo{\vect{1}} 
\newcommand\vei{\vect{\infty}} 
\newcommand\caP{\mathcal{P}} 
\newcommand\caC{\mathcal{C}} 
\newcommand\cbQ{\vect{\mathcal{Q}}} 
\newcommand\dcbP{\diag(\vect{\caP})} 
\newcommand\clOmper{\overline{\Omega_\upp}} 
\newcommand\FPH{F\"{o}ldes--Pol\'{a}\v{c}ik's Harnack } 

\usepackage{todonotes}
\newcommand{\red}[1]{#1}

\title[Long-time properties of KPP systems in space-time periodic media]{Persistence, extinction and spreading properties of non-cooperative Fisher--KPP systems in space-time periodic media}
\author{L\'{e}o Girardin}
\address[L. G.]{CNRS, Institut Camille Jordan, Universit\'{e} Claude Bernard Lyon-1, 43 boulevard du 11 novembre 1918, 69622 Villeurbanne Cedex, France}
\email{leo.girardin@math.cnrs.fr}

\begin{document}
\begin{abstract}
    This paper is concerned with asymptotic persistence, extinction and spreading properties for 
    non-cooperative Fisher--KPP systems with space-time periodic coefficients. Results are formulated in terms of
    a family of generalized principal eigenvalues associated with the linearized problem. When the maximal generalized 
    principal eigenvalue is negative, all solutions to the Cauchy problem become locally uniformly positive 
    in long-time, at least one space-time periodic uniformly positive entire solution exists, and solutions
    with compactly supported initial condition asymptotically spread in space at a speed given by a Freidlin--G\"{a}rtner-type formula.
    When another, possibly smaller, generalized principal eigenvalue is nonnegative, then on the 
    contrary all solutions to the Cauchy problem vanish uniformly and the zero solution is the unique 
    space-time periodic nonnegative entire solution. When the two generalized principal eigenvalues 
    differ and zero is in between, the long-time behavior depends on the decay at infinity of the initial condition. 
    \red{The proofs rely upon double-sided controls by solutions of cooperative systems. The control from below is new 
    for such systems and makes it possible to shorten the proofs and extend the generality of the system simultaneously.}
\end{abstract}

\maketitle

\section{Introduction}
This paper is concerned with reaction--diffusion systems of the form 
\begin{equation}\label{sys:KPP}\tag{KPP}
    \dcbP\veu =\veL\veu-\veC\veu\circ\veu,
\end{equation}
where: $\veu:(t_0,+\infty)\times\R^n\to\R^N$ is a vector-valued function of size $N\in\N^\star$, with a time variable 
$t\in(t_0,+\infty)\subset\R$ and a
space variable $x\in\R^n$, $t_0\in\overline{\R}$ and $n\in\N^\star$ being respectively the initial time and the spatial
dimension; each operator of the family $\vect{\caP}=(\caP_i)_{i\in[N]}$, with $[N]=\N\cap[1,N]$, 
has the form
\[
    \caP_i:u\mapsto\partial_t u -\nabla\cdot\left(A_i\nabla u\right)+q_i\cdot\nabla u,
\]
with $A_i:\R\times\R^n\to\R^{n\times n}$ and $q_i:\R\times\R^n\to\R^n$ periodic functions of 
$(t,x)$, respectively square matrix-valued and vector-valued;
$\veL,\veC:\R\times\R^n\to\R^{N\times N}$ are square matrix-valued periodic functions of
$(t,x)$; and $\circ$ denotes the Hadamard product between two vectors in $\R^N$.

We will study both nonnegative entire solutions of \eqref{sys:KPP}\red{, namely solutions defined in
$\R\times\R^n$ initiated at $t_0=-\infty$,} and solutions of the Cauchy problem initiated at $t_0=0$ and supplemented
with bounded nonnegative continuous initial conditions,
\begin{equation}
    \veu(0,x)=\veu_{\upini}(x)\quad\text{for all }x\in\R^n,\ \veu_{\upini}\in\caC_b(\R^n,\R^N),\ \veu_{\upini}\geq\vez.
    \label{IC}\tag{IC}
\end{equation}

The standing assumptions on $\vect{\caP}$, $\veL$ and $\veC$ are the following.
\begin{enumerate}[label=$({\mathsf{A}}_{\arabic*})$]
    \item \label{ass:ellipticity} The family $(A_i)_{i\in[N]}$ is \textit{uniformly elliptic}:
    \[
        0<\min_{i\in[N]}\min_{y\in\Sn}\min_{(t,x)\in\R\times\R^n}\left(y\cdot A_i(t,x)y\right).
    \]
    \item \label{ass:cooperative} The matrix $\underline{\veL}\in\R^{N\times N}$, whose entries are 
    \[
        \underline{l}_{i,j}=\min_{(t,x)\in\R\times\R^n}l_{i,j}(t,x)\quad\text{for all }(i,j)\in[N]^2,
    \]
    is \textit{essentially nonnegative}: its off-diagonal entries are nonnegative.
    \item \label{ass:irreducible} The matrix $\overline{\veL}\in\R^{N\times N}$, whose entries are 
    \[
        \overline{l}_{i,j}=\max_{(t,x)\in\R\times\R^n}l_{i,j}(t,x)\quad\text{for all }(i,j)\in[N]^2,
    \]
    is \textit{irreducible}: it does not have a stable subspace of the form
    $\vspan(\vect{e}_{i_1},\dots,\vect{e}_{i_k})$, where $k\in[N-1]$, $i_1,\dots,i_k\in[N]$ and $\vect{e}_i=(\delta_{ij})_{j\in[N]}$.
    By convention, $[0]=\emptyset$ and $1\times 1$ matrices are irreducible, even if zero.
    \item \label{ass:KPP} The matrix $\underline{\veC}\in\R^{N\times N}$, whose entries are 
    \[
        \underline{c}_{i,j}=\min_{(t,x)\in\R\times\R^n}c_{i,j}(t,x)\quad\text{for all }(i,j)\in[N]^2,
    \]
    is \textit{positive}: its entries are positive.
    \item \label{ass:smooth_periodic} There exists $\delta\in(0,1)$ such that
	$\veL,\veC\in\caC^{\delta/2,\delta}_\upp(\R\times\R^n,\R^{N\times N})$ and, 
	for each $i\in[N]$, $A_i\in\caC^{\delta/2,1+\delta}_\upp(\R\times\R^n,\R^{n\times n})$ and
    $q_i\in\caC^{\delta/2,\delta}_\upp(\R\times\R^n,\R^n)$. 
    Moreover, $A_i=A_i^\upT$ for each $i\in[N]$.
\end{enumerate}
The precise definition of the functional spaces appearing in \ref{ass:smooth_periodic} will
be clarified below, if not clear already. As usual in such a smooth and 
generic framework, the symmetry of the diffusion matrices in \ref{ass:smooth_periodic} is actually given for free.
\red{No symmetry assumption is made on $\veL$ and the irreducibility of $\overline{\veL}$ in \ref{ass:irreducible}
is equivalent to the irreducibility of the space-time average of $\veL$.}

\red{We are interested in long-time properties of solutions: persistence, extinction and propagation.
These have been studied before in the following special cases:
\begin{itemize}
    \item $N=1$, \textit{i.e.} the system reduces to a scalar equation of KPP type \cite{Nadin_2010,Berestycki_Hamel_Nadin};
    \item $N=2$, $n=1$, no advection and space-time homogeneous coefficients \cite{Griette_Raoul,Morris_Borger_Crooks};
    \item $n=1$, no advection and space-time homogeneous coefficients \cite{Girardin_2016_2};
    \item no advection, space-time homogeneous coefficients and essentially positive $\veL$ (instead of
	essentially nonnegative) \cite{Barles_Evans_S};
    \item $N=2$, $n=1$, $A_1=A_2=1$, no advection, space periodic coefficients and pointwise essentially positive $\veL$ \cite{Alfaro_Griette,Roques_Boutillon_Zamberletti_Papaix_2023};
    \item $N=2$, $n=1$, space periodic coefficients and pointwise essentially positive $\veL$ \cite{Griette_Matano_2021}.
\end{itemize}

The usual tool when studying long-time behaviors in reaction--diffusion models is the \textit{comparison principle}, 
namely the property that ordered initial conditions yield perpetually ordered solutions.
However, due to the positivity of $\veC$ (\textit{cf.} \ref{ass:KPP}), systems of the form \eqref{sys:KPP}
satisfy the comparison principle if and only if $N=1$. When $N\geq 2$, $\veL$ and $\veC$ can be chosen
in such a way that a comparison principle holds true in a neighborhood of $\vez$ or away from it 
\cite[Proposition 2.5]{Cantrell_Cosner_Yu_2018}, but such choices are not generic and never lead to a global comparison principle.
Systems that do not satisfy the comparison principle are referred to as \textit{non-monotone} or \textit{non-cooperative}.
In this paper, we present a unifying approach for non-cooperative KPP systems of arbitrary size $N\in\N^\star$, in 
arbitrary spatial dimension $n\in\N^\star$, with space-time periodic coefficients, with advection
and with minimal positivity requirements on $\veL$. 
To the best of our knowledge, this is the first time non-cooperative KPP systems
are studied at such a degree of generality. The techniques we bring together and develop are robust. They exploit
recent principal spectral theory results as well as the deep connection between non-cooperative KPP systems, KPP equations
and cooperative systems. Consequently, they could be deployed in other extensions of problems first stated for scalar KPP
equations, for instance the impact of different space-time heterogeneities \cite{Berestycki_Hamel_Nadin} or the logarithmic
delay of spreading fronts emanating from compactly supported initial conditions \cite{Hamel_Nolen_Ro}.}

For brevity, we will denote from now on 
\begin{equation*}
    \cbQ=\dcbP-\veL
\end{equation*}
the linear operator derived from the linearization of \eqref{sys:KPP} at $\veu=\vez$. By virtue of
the assumptions \ref{ass:cooperative}--\ref{ass:irreducible}, this linear operator is cooperative and fully coupled, and 
this will be a key property in the forthcoming analysis of the non-cooperative semilinear system \eqref{sys:KPP}.
\red{On the contrary, the nonlinear remainder of the reaction term plays a secondary role, and results can be generalized with 
minor technical adaptations to reaction terms of the form $\veL(t,x)\veu-\vect{f}\left(t,x,\veu\right)\circ\veu$
satisfying the following assumption of KPP type:
\begin{equation*}
    \forall(t,x)\in\R\times\R^n,\quad
    \begin{cases}
	\vect{f}(t,x,\vez)=\vez, \\
	\vect{f}(t,x,\vev)\geq\vez\quad\text{if }\vev\geq\vez.
    \end{cases}
\end{equation*}
Such a generalization can be found in our previous paper \cite{Girardin_2016_2}.
The motivation of the restriction in the present paper is twofold: on one hand, the form $\veL\veu-\veC\veu\circ\veu$ 
is sufficient for the applications we have in mind; on the other hand, it minimizes the verbosity and highlights
the role of the linear part.}

\subsection{Organization of the paper}
The remainder of Section 1 is devoted to a detailed introduction \red{and to the statement of the main results}. 
Section 2 \red{contains technical preliminaries. Section 3} contains the proofs.

\subsection{\red{Motivations}}
\red{Systems of the form \eqref{sys:KPP} arise naturally in population dynamics modeling when the population
under consideration has to be divided into discrete classes (age classes, phenotypic classes, size classes, etc.).}
Extensive \red{references and} discussions on these models can be
found in \cite{Girardin_2016_2,Girardin_Mazari_2022}. Here we only suggest briefly one example of application.

Elliott--Cornell \cite{Elliott_Cornel} introduced for modeling purposes the following model:
\begin{equation*}
    \begin{cases}
	\partial_t n_e = D_e\partial_{xx}n_e + r_e n_e(1-m_{ee}n_e -m_{ed}n_d)+\mu_d n_d-\mu_e n_e, \\
	\partial_t n_d = D_d\partial_{xx}n_d + r_d n_d(1-m_{de}n_e -m_{dd}n_d)+\mu_e n_e-\mu_d n_d.
    \end{cases}
\end{equation*}
This system was conceived as an eco-evolutionnary model for spatio-temporal \red{one-dimensional}
dynamics of a population with two phenotypes,
or morphs. Each morph $i\in\{e,d\}$ has a dispersal rate $D_i$, a growth rate $r_i$, is subjected to Lotka--Volterra competitive 
dynamics with intermorph and intramorph competition rates $m_{ij}$, and mutates into the other morph at rate $\mu_i$. 
All coefficients are space-time constants. The establisher morph $e$ is specialized
in growth, \textit{i.e.} $r_e>r_d$, whereas the disperser morph $d$ is specialized in dispersal, \textit{i.e.} $D_d>D_e$.

The role of each morph during a population invasion was investigated, heuristically in \cite{Elliott_Cornel} and
subsequently rigorously in \cite{Morris_Borger_Crooks}.

The system above is a special case of \eqref{sys:KPP}.
\red{Our results make it possible to study several natural extensions of the model. 
Spatial dimensions larger than $1$ are natural for applications, in particular $n=2$ or $n=3$ for population dynamics.
Numbers of morphs larger than $N=2$ can help to model more precisely the phenotypic landscape, or even to discretize
a continuous model.}
Advection terms can be added to model, \textit{e.g.}, directional wind.
Temporal periodicity of the coefficients can be added to model, depending on the timescale, seasonality or nychthemeral rhythms.
Spatial periodicity of the coefficients can be added to model biological invasions in periodic landscapes, \textit{e.g.}, vineyards.
More generally, spatiotemporal periodicity is a way of introducing environmental heterogeneity while keeping strong 
mathematical tractability. \red{Finally, since the assumption \ref{ass:irreducible} does not require pointwise
space-time irreducibility but only irreducibility on average, it is also possible to analyze models where, \textit{e.g.},
mutations occur only during short outbursts triggered by time periodic events that are out of phase, in such a way 
that the time periodic functions $\mu_e,\mu_d\geq 0$ satisfy $\mu_e\neq0$, $\mu_d\neq 0$, $\mu_e\mu_d=0$.}

\subsection{\red{Notations}}

Generally speaking, notations are chosen consistently with our previous paper on space-time
homogeneous coefficients \cite{Girardin_2016_2} and with our paper with I. Mazari on the principal
spectral analysis of $\cbQ$ \cite{Girardin_Mazari_2022}.

In the whole paper, $\N$ is the set of nonnegative integers, which contains $0$.

We fix once and for all $n+1$ positive numbers $T, L_1, \dots, L_n\in \R_+^\star$. For the sake of brevity, we use the notations
$L=(L_1,\dots,L_n)$, $(0,L)=(0,L_1)\times\dots\times(0,L_n)$ and $|[0,L]|=\prod_{\alpha=1}^n L_\alpha$. Unless specified otherwise, 
time and space periodicities refer to, respectively, $T$-periodicity with respect to $t$ and $L_\alpha$-periodicity with respect to 
$x_\alpha$ for each $\alpha\in[n]$ (or $L$-periodicity with respect to $x$ for short). The space-time periodicity cell $(0,T)\times(0,L)$ 
is denoted $\Omega_\upp$ and its volume is $T|[0,L]|$.

Vectors in $\R^N$ and matrices in $\R^{N\times N}$ are denoted in bold font. 
Functional operators are denoted in calligraphic typeface (bold if they act on functions valued in $\R^N$).
Functional spaces, \textit{e.g.} $\mathcal{W}^{1,\infty}(\R\times\R^n,\R^N)$, 
are also denoted in calligraphic typeface. A functional space $\mathcal{X}$ denoted with a subscript $\mathcal{X}_\upp$, 
$\mathcal{X}_{t-\upp}$ or $\mathcal{X}_{x-\upp}$ is restricted to functions that are space-time periodic, time periodic or space periodic respectively.

For clarity, H\"{o}lder spaces of functions with $k\in\mathbb{N}$ derivatives that are all
H\"{o}lder-continuous with exponent $\alpha\in(0,1)$ are denoted $\caC^{k+\alpha}$; when the domain is $\R\times\R^n$, it 
should be unambiguously understood that $\caC^{k+\alpha,k'+\alpha'}$ denotes the set of functions that have $k$ 
$\alpha$-H\"{o}lder-continuous derivatives in time and $k'$ $\alpha'$-H\"{o}lder-continuous derivatives in space.

For any two vectors $\veu,\vev\in\R^N$, $\veu\leq\vev$ means $u_i\leq v_i$ for all 
$i\in[N]$, $\veu<\vev$ means $\veu\leq\vev$ together with $\veu\neq\vev$ and $\veu\ll\vev$ means $u_i<v_i$ for all $i\in[N]$. If 
$\veu\geq\vez$, we refer to $\veu$ as \textit{nonnegative}; if $\veu>\vez$, as \textit{nonnegative nonzero}; if $\veu\gg\vez$, as 
\textit{positive}. The sets of all nonnegative, nonnegative nonzero, positive vectors are respectively denoted $[\vez,\vei)$,
$[\vez,\vei)\backslash\{\vez\}$ and $(\vez,\vei)$. The vector whose entries are all equal to $1$ is denoted $\veo$ and this never refers to
an indicator function.
Similar notations and terminologies might be used in other dimensions and for matrices. The identity matrix is denoted $\vect{I}$.

Similarly, a function can be nonnegative, nonnegative nonzero, positive. For clarity, a positive function is a function with only positive values.

To avoid confusion between operations in the state space $\R^N$ and operations in the spatial domain $\R^n$, 
Latin indexes $i,j,k$ are assigned to vectors and matrices of size $N$ whereas Greek indexes $\alpha,\beta,\gamma$ 
are assigned to vectors and matrices of size $n$. 
We use mostly subscripts to avoid confusion with algebraic powers, but when both Latin and Greek indexes are involved, we 
move the Latin ones to a superscript position, \textit{e.g.} $A^i_{\alpha,\beta}(t,x)$.
We denote scalar products in $\R^N$ with the transpose operator, $\veu^\upT\vev=\sum_{i=1}^N u_i v_i$,
and scalar products in $\R^n$ with a dot, $x\cdot y =\sum_{\alpha=1}^n x_\alpha y_\alpha$. 

For any vector $\veu\in\R^N$, $\diag(\veu)$, $\diag(u_i)_{i\in[N]}$ or $\diag(u_i)$ for short refer to the diagonal matrix in $\R^{N\times N}$
whose $i$-th diagonal entry is $u_i$. These notations can also be used if $\veu$ is a function valued in $\R^N$.

Finite dimensional Euclidean norms are denoted $|\cdot |$ whereas the notation $\|\cdot \|$ is reserved for norms in functional spaces.

The notation $\circ$ is reserved in the paper for the Hadamard product (component-wise product of vectors or matrices)
and never refers to the composition of functions.
For any vector $\vev\in\R^N$ and $p\in\R$, $\vev^{\circ p}$ denotes the vector $(v_i^p)_{i\in[N]}$.

\subsection{Results}

Before stating the results, we need to introduce a family of generalized principal eigenvalues that was previously
studied in \cite{Girardin_Mazari_2022}. The family $\left(\lambda_{1,z}\right)_{z\in\R^n}$ is defined by:
\begin{equation}
    \lambda_{1,z} = \lambda_{1,\upp}\left( \upe_{-z}\cbQ\upe_z \right),
\end{equation}
where $\lambda_{1,\upp}$ denotes the periodic principal eigenvalue classically given by the Krein--Rutman theorem
and where $\upe_{\pm z}:x\in\R^n\mapsto\upe^{\pm z\cdot x}$. The operator $\upe_{-z}\cbQ\upe_z$ can be alternatively written as:
\begin{equation}
    \upe_{-z}\cbQ\upe_z = \cbQ-\diag\left(2A_i z\cdot\nabla+z\cdot A_i z+\nabla\cdot\left(A_i z\right)-q_i\cdot z\right).
\end{equation}
For any $z\in\R^n$, there exists a unique, up to multiplication by a positive constant, positive periodic principal eigenfunction 
$\veu_z\in\caC^{1,2}_{\upp}(\R\times\R^n,(\vez,\vei))$ satisfying $\cbQ(\upe_z\veu_z)=\lambda_{1,z}\upe_z\veu_z$.

Recall that $z\in\R^n\mapsto\lambda_{1,z}$ is strictly concave, coercive, with one global maximum.
We denote:
\begin{equation}
    \lambda_1=\max_{z\in\R^n}\lambda_{1,z}\quad\text{and}\quad\lambda_1'=\lambda_{1,0}.
\end{equation}
The equality $\lambda_1=\lambda_1'$ can be true or false. 
\red{The strict inequality $\lambda_1'<\lambda_1$ can be induced by, \textit{e.g.}, nonzero advection rates $q_i$ or 
spatial heterogeneities combined with asymmetries in the matrix $\veL$ \cite{Girardin_Mazari_2022,Griette_Matano_2021}.}

As was proved in \cite{Girardin_Mazari_2022}, $\lambda_1$ and $\lambda_1'$ can be alternatively defined as:
\begin{equation}
    \lambda_1 = \sup\left\{ \lambda\in\R\ |\ \exists \veu\in\caC^{1,2}_{t-\upp}(\R\times\R^n,(\vez,\vei))\ \cbQ\veu\geq\lambda\veu \right\},
\end{equation}
\begin{equation}
    \lambda_1' = \inf\left\{ \lambda\in\R\ |\ \exists \veu\in\mathcal{W}^{1,\infty}\cap\caC^{1,2}_{t-\upp}(\R\times\R^n,(\vez,\vei))\ \cbQ\veu\leq\lambda\veu \right\}.
\end{equation}

We are now in a position to state our results.

The first result states a condition for the uniform extinction of any solution of the Cauchy problem.

\begin{thm}\label{thm:extinction}
    Assume $\lambda_1'\geq 0$. 

    Then all solutions $\veu$ of the Cauchy problem \eqref{sys:KPP}--\eqref{IC} satisfy
    \begin{equation}
	\lim_{t\to+\infty}\max_{i\in[N]}\sup_{x\in\R^n}u_i (t,x)=0.
	\label{eq:extinction}
    \end{equation}
\end{thm}

As an immediate corollary, when $\lambda_1'\geq 0$, $\vez$ is the only nonnegative bounded 
entire solution of \eqref{sys:KPP}.
This is no longer true when $\lambda_1'<0$, as stated by the following theorem.

\begin{thm}\label{thm:existence_persistent_entire_solution}
    Assume $\lambda_1'<0$. 

    Then there exists a uniformly positive space-time periodic entire solution $\veu^\star$ of \eqref{sys:KPP}.
\end{thm}

Actually, when $\lambda_1'<0$ and $z\in\R^n\mapsto\lambda_{1,z}$ is not maximal at $z=0$, 
all solutions of the Cauchy problem starting from sufficiently large initial conditions persist locally
uniformly, as stated in the next result.

\begin{thm}\label{thm:persistence_large_solutions}
    Assume $\lambda_1'<0$ and the existence of $z\in\R^n$ such that:
    \begin{enumerate}
	\item $\lambda_{1,z}<0$;
	\item $\zeta\in(0,2)\mapsto\lambda_{1,\zeta z}$ is increasing in a neighborhood of $1$;
	\item there exists $C>0$, $B\in\R$ such that, for all $x\in\R^n$ such that $z\cdot x\leq B$, 
	    $\min_{i\in[N]}u_{\upini,i}(x)\geq C^{-1}\upe_{z}(x)$.
    \end{enumerate}

    Then the solution $\veu$ of the Cauchy problem \eqref{sys:KPP}--\eqref{IC} satisfies
    \begin{equation}
	\liminf_{t\to+\infty}\min_{i\in[N]}\inf_{|x|\leq R}u_i(t,x)>0\quad\text{for all }R>0.
	\label{eq:Cauchy_persistence}
    \end{equation}
\end{thm}

As stated by the following result, the stronger condition $\lambda_1<0$ is sufficient to ensure 
the locally uniform persistence of any nonzero solution of the Cauchy problem. This type of 
property is usually referred to as a \textit{hair-trigger effect}. Note that if $\lambda_1'<0$
and $z\in\R^n\mapsto\lambda_{1,z}$ is maximal at $z=0$, then obviously $\lambda_1<0$, so that
in all cases $\lambda_1'<0$ implies the persistence of at least some solutions.

\begin{thm}\label{thm:hair_trigger_effect}
    Assume $\lambda_1<0$. 

    Then all solutions $\veu$ of the Cauchy problem \eqref{sys:KPP}--\eqref{IC} \red{with nonzero initial conditions
    $\veu_{\upini}$} satisfy \eqref{eq:Cauchy_persistence}.
\end{thm}

The next result shows that when $\lambda_1\geq 0$, solutions of the Cauchy problem starting from sufficiently small 
initial conditions go extinct locally uniformly. This violation of the hair-trigger effect is especially interesting in the 
intermediate case $\lambda_1'<0\leq\lambda_1$. 

\begin{thm}\label{thm:extinction_small_solutions}
    Assume $\lambda_1\geq 0$ and the existence of $z\in\R^n$ such that:
    \begin{enumerate}
	\item $\lambda_{1,z}\geq 0$;
	\item there exists $C>0$ such that, for all $x\in\R^n$, $\max_{i\in[N]}u_{\upini,i}(x)\leq C\upe_z(x)$.
    \end{enumerate}

    Then the solution $\veu$ of the Cauchy problem \eqref{sys:KPP}--\eqref{IC} satisfies
    \begin{equation}
	\lim_{t\to+\infty}\max_{i\in[N]}\sup_{|x|\leq R}u_i (t,x)= 0\quad\text{for all }R>0.
	\label{eq:local_extinction}
    \end{equation}
\end{thm}

This collection of results indicates in particular that solutions evolving from compactly supported initial conditions persist when
$\lambda_1<0$ and go extinct at least locally uniformly when $\lambda_1\geq 0$. It becomes then natural 
to investigate spreading properties in the case $\lambda_1<0$. The last result provides a Freidlin--G\"{a}rtner-type formula
\cite{Gartner_Freidlin} for the asymptotic spreading speed of such solutions. 

We introduce, for any direction $e\in\mathbb{S}^{n-1}$ and any decay rate $\mu>0$:
\begin{equation}
    c^\mu_e = \frac{\lambda_{1,-\mu e}}{-\mu},\quad c^\star_e=\min_{\mu>0}c^\mu_e,\quad c^{\upFG}_e = \min_{\substack{e'\in\mathbb{S}^{n-1}\\e\cdot e'>0}}\frac{c^\star_{e'}}{e\cdot e'}.
    \label{eq:spreading_speed_direction_e}
\end{equation}
The fact that the minima involved in the definition of $c^\star_e$ and $c^{\upFG}_e$ are well-defined is classical
in KPP-type problems and will be verified later on.

\begin{thm}\label{thm:FG_formula}
    Assume $\lambda_1<0$. 

    Then all solutions $\veu$ of the Cauchy problem \eqref{sys:KPP}--\eqref{IC} \red{with nonzero and compactly
    supported initial conditions $\veu_{\upini}$}
    spread in the direction $e\in\mathbb{S}^{n-1}$ at speed $c^{\upFG}_e$, namely
    \begin{equation}
	\liminf_{t\to+\infty}\min_{i\in[N]}\inf_{|x|\leq R}u_i(t,x+cte)>0\quad\text{for all }R>0\text{ and }c\in(0,c^{\upFG}_e),
	\label{eq:spreading_subestimate}
    \end{equation}
    \begin{equation}
	\lim_{t\to+\infty}\max_{i\in[N]}\sup_{|x|\leq R}u_i(t,x+cte)=0\quad\text{for all }R>0\text{ and }c>c^{\upFG}_e.
	\label{eq:spreading_superestimate}
    \end{equation}
\end{thm}

\subsection{Comments}

The first two results show that the sign of $\lambda_1'$ is a sharp criterion for the existence of nonnegative nonzero
entire solutions. However, when studying the long-time behavior of the Cauchy problem, the knowledge of the sign of $\lambda_1$
is also needed, and moreover in the case $\lambda_1'<0\leq\lambda_1$ the outcome depends also on the initial condition. 

The case $\lambda_1=0>\lambda_1'$ was stated by Nadin as an open problem in the scalar case \cite{Nadin_2010}, but is actually
within reach with the same methods. Since our paper covers the scalar case as the particular case $N=1$, Theorem 
\ref{thm:extinction_small_solutions} solves the question raised by \cite{Nadin_2010}.

The sharpness of the conditions on the size of (the exponential decay of) the initial condition in Theorems 
\ref{thm:persistence_large_solutions} and \ref{thm:extinction_small_solutions} can be discussed as in \cite[p. 1296]{Nadin_2010}.
We only mention that uniformly positive initial conditions will always satisfy the condition of Theorem 
\ref{thm:persistence_large_solutions} while compactly supported initial conditions will always satisfy that of 
Theorem \ref{thm:extinction_small_solutions}.

Heuristically, the results can be summarized as follows.
Since we restrict ourselves \textit{a priori} to bounded initial conditions, that is, to 
bounded perturbations of $\vez$, the family $(\lambda_{1,z})$ gives stability criteria
depending on the exponential decay $z$ at spatial infinity. When $\vez$ is unstable with respect 
to any exponential decay, that is, $\lambda_1<0$, then it is actually unstable with respect to 
compactly supported initial conditions and the hair-trigger effect holds. But when
$\lambda_1>0>\lambda_1'$, then some exponential decays are too strong and make $\vez$ stable:
the hair-trigger effect does not hold. For instance, for the scalar operator 
$\cbQ=\partial_t-\partial_{xx}+\partial_x-1/8$, the generalized principal eigenvalues satisfy 
$\lambda_{1,z}=z(1-z)-1/8$, and values of $z$ satisfying the monotonicity condition stated in 
Theorem \ref{thm:persistence_large_solutions} are $z\in(0,1/2)$. In this interval, the sign change
occurs at $z^\star = \left( 1-\sqrt{2}/2 \right)/2$. Applying Theorem
\ref{thm:persistence_large_solutions} with $z=z^\star-\varepsilon$ and Theorem 
\ref{thm:extinction_small_solutions} with $z=z^\star+\varepsilon$, we find that the zero 
steady state is, with respect to perturbations of the form $C\min(\upe_{z},\upe_{z'})$
with $z'\leq 0\leq z$:
\begin{itemize}
    \item unstable if $z\in[0,z^\star)$;
    \item stable if $z>z^\star$.
\end{itemize}
Interestingly, this confirms the crucial role of the monotonicity condition of Theorem 
\ref{thm:persistence_large_solutions}: the stability of the zero steady state is fully determined
by an arbitrarily small open neighborhood of $z^\star$, and in particular the other sign change
of $z\mapsto \lambda_{1,z}$ at $z=\left( 1+\sqrt{2}/2 \right)/2$ brings no additional stability
information.
Intuitively, the transport of the solution at speed $1$ towards the right, encoded in 
the advection term $+\partial_x$, washes away all solutions having an initial leftward 
exponential decay $z>\left( 1-\sqrt{2}/2 \right)/2$, that is, solutions whose initial
conditions are too thin-tailed towards the left. 

The Freidlin--G\"{a}rtner-type formula of Theorem \ref{thm:FG_formula} was established in the
space-time periodic scalar case ($N=1$) by Berestycki--Hamel--Nadin \cite[Theorem 1.13]{Berestycki_Hamel_Nadin}. For space-time 
periodic cooperative systems satisfying appropriate assumptions, it was established recently by
Du--Li--Shen\cite{Du_Li_Shen_2022}\red{, with the help of a monotone recursion method due to Weinberger \cite{Weinberger_1982}
that has had considerable impact in reaction--diffusion theory for more than four decades.
However, this method cannot be applied directly to non-cooperative KPP systems.} 
Our arguments of proof are very similar to those used to prove the other theorems\red{: the problem is reformulated in terms of
super- and sub-estimates, and solved thanks to double-sided controls by solutions of cooperative systems}.
Roughly speaking, asymptotic spreading results are still
stability properties and can be understood as persistence/extinction results in moving frames. 
With similar methods, the spreading speed of solutions evolving from exponentially decaying initial values could be investigated. 
Here we focus on the compactly supported case, which is more relevant biologically and usually provides in KPP-type problems
sub-estimates for more general invasions -- it will be clear from the proof that this is again the case here, 
despite the lack of comparison principle.
On the contrary, the construction of entire solutions that describe the invasion of open space by 
positive population densities at constant speed, namely \textit{pulsating traveling waves} \cite{Du_Li_Shen_2022,Nadin_2009},
is \red{an existence problem and not a super- or sub-estimate problem. Due to this fundamental difference,} 
it requires other methods. This problem will be investigated in a future sequel. There, we will prove in particular 
that $c^\star_e$ is the minimal wave speed of planar pulsating traveling waves in the direction $e$, as is standard in KPP-type problems.

Let us mention that, by drawing inspiration from \cite{Girardin_Mazari_2022} and \cite{Nadin_2011}, we 
could combine elementarily results on the dependence of the generalized principal eigenvalues on the coefficients
and the Freidlin--G\"{a}rtner formula to obtain dependence results for the spreading speed. We 
point out in particular that the spreading speed is in general not monotonic with respect to the diffusivity 
amplitude \cite{Nadin_2011}, but is monotonic with respect to the matrix entries $l_{i,j}$ \cite{Girardin_Mazari_2022}. 
We also point out that space homogeneity and time homogeneity of the coefficients, supplemented with appropriate specific 
conditions, lead to simplifications of the formula or to upper or lower estimates \cite{Girardin_Mazari_2022}.

The results we manage to prove in the present paper are analogous to their scalar counterparts 
\cite{Nadin_2010,Berestycki_Hamel_Nadin}
but their proofs are carefully improved \red{in such a way that the comparison principle is only applied on cooperative
systems with the same linear part as \eqref{sys:KPP} but with a modified nonlinear part}. 
This is the main difficulty and novelty of this work. It is actually known that not all results of the scalar
case can be extended in this way; in particular, Liouville-type results on the uniformly positive entire solution
are in general false even with constant coefficients
\cite{Girardin_2017,Girardin_2018,Girardin_Griette_2020,Morris_Borger_Crooks,Cantrell_Cosner_Yu_2018}.
In this regard, our intent is precisely to show what can be extended from the case $N=1$ to the general case, and what cannot.
\red{It is also known \cite[Section 1.4.1]{Girardin_2016_2} that there are simple and application-wise natural examples
of non-cooperative KPP systems that cannot be approximated from below globally in time by cooperative systems with the same linear part. 
In this paper, we overcome this obstacle thanks to space-time global bounds, Harnack inequalities, Gaussian estimates 
and by considering only large enough times. This approach is new for such systems and is inspired by recent works on
nonlocal equations \cite{Bouin_Henderso-1}. Other nonlocal equations that are continuous versions of non-cooperative KPP systems 
\cite{Griette_2017} could be analyzed efficiently with the same techniques.}

\section{\red{Preliminaries}}
In this section, we establish or recall basic results that will be used repeatedly in the main proofs.

\subsection{Global boundedness estimates and absorbing set}

In this subsection, we establish that a solution of \eqref{sys:KPP}--\eqref{IC} satisfies
a global boundedness estimate that depends only on the initial values, and that
becomes uniform with respect to the initial values in long time. This is a direct
adaptation of the proof of \cite[Theorem 1.2]{Girardin_2016_2}, shortened
by the stronger assumption \ref{ass:KPP}.

\begin{prop}\label{prop:global_bound_absorbing_set}
    There exists a constant $K>0$ such that, for any solution $\veu$ of \eqref{sys:KPP}--\eqref{IC},
    \begin{equation*}
	\veu\leq \left(K+\sup_{x\in\R^n}\max_{i\in[N]}u_{\upini,i}(x)\right)\veo \quad\text{in }[0,+\infty)\times\R^n,
    \end{equation*}
    \begin{equation*}
	\limsup_{t\to+\infty}\sup_{x\in\R^n}\max_{i\in[N]}u_i(t,x)\leq K.
    \end{equation*}
\end{prop}

\begin{proof}
    By assumptions \ref{ass:smooth_periodic} and \ref{ass:KPP}, there exist constants
    $r,K>0$ such that, for any $\veu\geq\vez$,
    \begin{equation*}
	\veL\veu-\veC\veu\circ\veu\leq r\left( \veo^\upT\veu \right)\left( K\veo-\veu \right).
    \end{equation*}
    In particular, solutions $\veu$ of \eqref{sys:KPP}--\eqref{IC} satisfy
    $\dcbP\veu\leq r(\veo^\upT\veu)\left( K\veo-\veu \right)$, that is: 
    \begin{equation*}
	\caP_i u_i \leq r(K-u_i)\sum_{j=1}^N u_j\quad\text{for each }i\in[N].
    \end{equation*}
    Whenever $u_i\geq K$, $\caP_i u_i\leq ru_i(K-u_i)$.
    In particular, $\underline{u}_i=\max(u_i,K)$ is a sub-solution of the equation
    $\caP_i u = ru(K-u)$. Moreover it satisfies $\underline{u}_i(0,\cdot)\leq M$, where:
    \begin{equation*}
	M=\sup_{x\in\R^n}\max_{i\in[N]}u_{\upini}(x). 
    \end{equation*}

    Now, still for the equation $\caP_i u = ru(K-u)$, consider the space-homogeneous 
    (super-)solution $\overline{u}_i$ satisfying
    \begin{equation*}
	\begin{cases}
	    \caP_i\overline{u}_i=r\overline{u}_i(K-\overline{u}_i) & \text{in }(0,+\infty)\times\R^n, \\
	    \overline{u}_i(0,\cdot)=K+M & \text{in }\R^n.
	\end{cases}
    \end{equation*}

    By virtue of the scalar comparison principle \cite{Protter_Weinberger}, 
    $\underline{u}_i\leq\overline{u}_i$ globally in
    $[0,+\infty)\times\R^n$ for each $i\in[N]$. Since $\overline{u}_i\leq K+M$ and 
    $\lim_{t\to+\infty}\overline{u}_i=K$, this ends the proof.
\end{proof}

As a consequence, we obtain the following corollary on entire solutions. In particular,
this applies to space-time periodic solutions.

\begin{cor}\label{cor:global_bound_entire}
    All nonnegative globally bounded entire solutions $\veu$ of \eqref{sys:KPP} satisfy
    \begin{equation*}
	\veu\leq K\veo\quad\text{in }\R\times\R^n.
    \end{equation*}
\end{cor}

\subsection{\red{Harnack inequality for linear cooperative systems}}

\red{For self-containment and ease of reading, we recall here the Harnack inequality 
that we proved with I. Mazari in \cite{Girardin_Mazari_2022}. It is a 
refinement of \FPH inequality \cite[Theorem 3.9]{Foldes_Polacik_2009}
for parabolic cooperative systems and Arapostathis--Ghosh--Marcus's Harnack inequality 
\cite[Theorem 2.2]{Araposthathis_} for elliptic cooperative systems.

We denote by $\sigma>0$ the smallest positive entry of $\overline{\veL}$ (\textit{cf.} \ref{ass:irreducible})
and by $K\geq 1$ the smallest positive number such that
\[
    K^{-1}\leq\min_{i\in[N]}\min_{y\in\Sn}\min_{(t,x)\in\clOmper}\left(y\cdot A_i(t,x)y\right),
\]
\[
    \max_{i\in[N]}\max_{y\in\Sn}\max_{(t,x)\in\clOmper}\left(y\cdot A_i(t,x)y\right)\leq K,
\]
\[
    \max_{i\in[N]}\max_{\alpha\in[n]}\max_{(t,x)\in\clOmper}|q_{i,\alpha}(t,x)|\leq K,
\]
\[
    \max_{i,j\in[N]}\sup_{(t,x)\in\clOmper}|l_{i,j}(t,x)|\leq K.
\]
The existence of $K$ is given by \ref{ass:ellipticity} and \ref{ass:smooth_periodic}.

\begin{prop}[Proposition 2.4 in \cite{Girardin_Mazari_2022}]\label{prop:harnack_inequality}
Let $\theta\geq\max\left(T,L_1,\dots,L_n\right)$ and $\vect{f}\in\mathcal{L}^\infty\cap\caC^{\delta/2,\delta}(\R\times\R^n,\R^N)$, 
where $\delta\in(0,1)$ is as in \ref{ass:smooth_periodic}. Let $F>0$ such that  
\[
    \max_{i\in[N]}\sup_{(t,x)\in\R\times\R^n}|f_i(t,x)|\leq F.
\]

There exists a constant $\overline{\kappa}_{\theta,F}>0$, determined only by $n$, $N$, $\sigma$, $K$ and the parameters
$\theta$ and $F$ such that, if $\veu\in\caC([-2\theta,6\theta]\times[-\frac{3\theta}{2},\frac{3\theta}{2}]^n,[\vez,\vei))$ 
is a solution of $\cbQ\veu=\diag(\vect{f})\veu$, then
\[
    \min_{i\in[N]}\min_{(t,x)\in[5\theta,6\theta]\times[-\frac{\theta}{2},\frac{\theta}{2}]^n}u_i(t,x)
    \geq \overline{\kappa}_{\theta,F}\max_{i\in[N]}\max_{(t,x)\in[0,2\theta]\times[-\frac{\theta}{2},\frac{\theta}{2}]^n}u_i(t,x).
\]
\end{prop}

We will apply repeatedly this Proposition in what follows. The global boundedness estimates of Proposition
\ref{prop:global_bound_absorbing_set} will be key since they make it possible to set $\vect{f}=\veC\veu$.}

\subsection{\red{Comparison principle for semilinear cooperative systems}}

\red{Similarly, we recall below a form of strong comparison principle for semilinear systems. 
It can be derived easily as a consequence of the classical result of Protter--Weinberger
\cite[Chapter 3, Theorem 13]{Protter_Weinberger} and of 
the above Harnack inequality; the proof is deliberately not detailed.

Note that the assumptions below could be relaxed in several ways; more general statements can be found in the literature. 
Here we only state with minimal verbosity a comparison principle sufficient for our purposes.

\begin{prop}\label{prop:comparison_principle}
    Let $f_1,f_2,\dots,f_N\in\caC^1(\R,\caC^{\delta/2,\delta}_{\upp}(\R\times\R^n,\R))$, where (with a slight abuse of notation)
    $f_i:v\mapsto \left[(t,x)\mapsto f_i(t,x,v)\right]$ for each $i\in[N]$ and where $\delta\in(0,1)$ is as in 
    \ref{ass:smooth_periodic}. 

    Let $\overline{\veu},\underline{\veu}\in\caC^{1,2}((0,+\infty)\times\R^n,[\vez,\vei))\cap\caC_b([0,+\infty)\times\R^n)$ 
    such that
    \begin{equation*}
        \begin{cases}
	    \cbQ\overline{\veu}\geq\left(f_i(\overline{u}_i)\right)_{i\in[N]} & \text{in }(0,+\infty)\times\R^n, \\
	    \cbQ\underline{\veu}\leq\left(f_i(\underline{u}_i)\right)_{i\in[N]} & \text{in }(0,+\infty)\times\R^n, \\
	    \underline{\veu}(0,\cdot)\leq\overline{\veu}(0,\cdot) & \text{in }\R^n.
        \end{cases}
    \end{equation*}
    
    Then $\underline{\veu}\leq\overline{\veu}$ globally in $[0,+\infty)\times\R^n$. 
    
    Furthermore, if there exist $t^\star>T$, $x^\star\in\R^n$ and $i^\star\in[N]$ such that 
    $\overline{u}_{i^\star}(t^\star,x^\star)=\underline{u}_{i^\star}(t^\star,x^\star)$, then 
    $\underline{\veu}=\overline{\veu}$ globally in $[0,+\infty)\times\R^n$.
\end{prop}

Below we state a result useful for manipulating unbounded super- or sub-solutions.

\begin{prop}\label{prop:generalized_comparison_principle}
    Let $f_1,f_2,\dots,f_N\in\caC^1(\R,\caC^{\delta/2,\delta}_{\upp}(\R\times\R^n,\R))$, where (with a slight abuse of notation)
    $f_i:v\mapsto \left[(t,x)\mapsto f_i(t,x,v)\right]$ for each $i\in[N]$ and where $\delta\in(0,1)$ is as in 
    \ref{ass:smooth_periodic}. 

    Let $\overline{\vev},\overline{\vew},\underline{\vev},\underline{\vew}\in
    \caC^{1,2}((0,+\infty)\times\R^n,[\vez,\vei))\cap\caC([0,+\infty)\times\R^n)$, possibly unbounded.
    Let
    \begin{equation*}
	\overline{\veu}=\left( \min(\overline{v}_i,\overline{w}_i) \right)_{i\in[N]},\quad
	\underline{\veu}=\left( \max(\underline{v}_i,\underline{w}_i) \right)_{i\in[N]}.
    \end{equation*}
    Assume:
    \begin{equation*}
        \begin{cases}
	    \cbQ\overline{\vev}\geq\left(f_i(\overline{v}_i)\right)_{i\in[N]} & \text{in }(0,+\infty)\times\R^n, \\
	    \cbQ\overline{\vew}\geq\left(f_i(\overline{w}_i)\right)_{i\in[N]} & \text{in }(0,+\infty)\times\R^n, \\
	    \cbQ\underline{\vev}\leq\left(f_i(\underline{v}_i)\right)_{i\in[N]} & \text{in }(0,+\infty)\times\R^n, \\
	    \cbQ\underline{\vew}\leq\left(f_i(\underline{w}_i)\right)_{i\in[N]} & \text{in }(0,+\infty)\times\R^n, \\
	    \underline{\veu}(0,\cdot)\leq\overline{\veu}(0,\cdot) & \text{in }\R^n, \\
	    \underline{\veu},\overline{\veu}\in\caC_b([0,+\infty)\times\R^n).
        \end{cases}
    \end{equation*}

    Then the conclusions of Proposition \ref{prop:comparison_principle} remain true.
\end{prop}

\begin{proof}
    We only prove that $\overline{\veu}$ is a so-called generalized super-solution, namely it satisfies
    the differential inequality $\cbQ\overline{\veu}\geq(f_i(\overline{u}_i))_{i\in[N]}$ in some weak Sobolev sense.
    Then it follows from symmetric arguments that $\underline{\veu}$ is a generalized sub-solution,
    and then the conclusion follows by applying an appropriate version of the maximum principle in Sobolev spaces.

    Let $i\in[N]$. Let $(t,x)$ such that $\overline{u}_i(t',x')=\overline{v}_i(t',x')$ for all $(t',x')$ in an open
    neighborhood of $(t,x)$. Then:
    \begin{align*}
	(\caP_i\overline{u}_i)(t,x) & = (\caP_i\overline{v}_i)(t,x) \\
	& \geq \sum_{j=1}^N l_{i,j}(t,x)\overline{v}_j(t,x)+f_i(t,x,\overline{v}_i(t,x)) \\
	& \geq \sum_{j\in[N]\backslash\{i\}} l_{i,j}(t,x)\overline{v}_j(t,x)+l_{i,i}(t,x)\overline{v}_i(t,x)+f_i(t,x,\overline{v}_i(t,x)) \\
	& \geq \sum_{j\in[N]\backslash\{i\}} l_{i,j}(t,x)\min(\overline{v}_j,\overline{w}_j)(t,x)+l_{i,i}(t,x)\overline{u}_i(t,x)+f_i(t,x,\overline{u}_i(t,x)) \\
	& \geq \sum_{j=1}^N l_{i,j}(t,x)\overline{u}_j(t,x)+f_i(t,x,\overline{u}_i(t,x))
    \end{align*}
    where the essential nonnegativity of $\veL$, \textit{cf.} \ref{ass:cooperative}, was used to deal with the possibility
    that there exist $j_1,j_2\in[N]\backslash\{i\}$ such that $\overline{u}_{j_1}(t,x)=\overline{v}_{j_1}(t,x)$ while
    $\overline{u}_{j_2}(t,x)=\overline{w}_{j_2}(t,x)$.

    By reversing the roles of $\overline{\vev}$ and $\overline{\vew}$, we obtain the same differential inequality in open
    space-time subsets where $\overline{u}_i=\overline{w}_i$. Since 
    \begin{equation*}
	\{ (t,x)\in(0,+\infty)\times\R^n\ |\ \overline{u}_i=\overline{v}_i\}
	\cup
	\{ (t,x)\in(0,+\infty)\times\R^n\ |\ \overline{u}_i=\overline{w}_i\}
	=(0,+\infty)\times\R^n,
    \end{equation*}
    the same differential inequality is actually true in the open set
    \begin{equation*}
	\left[(0,+\infty)\times\R^n\right]\backslash\left[ \partial\{\overline{u}_i=\overline{v}_i\}\cap\partial\{\overline{u}_i=\overline{w}_i\} \right].
    \end{equation*}

    Since this is true for any $i\in[N]$, the inequality $\cbQ\overline{\veu}\geq(f_i(\overline{u}_i))_{i\in[N]}$
    is true (in classical sense) in the open set
    \begin{equation*}
	\left[(0,+\infty)\times\R^n\right]\backslash\bigcup_{i=1}^N\left[ \partial\{\overline{u}_i=\overline{v}_i\}\cap\partial\{\overline{u}_i=\overline{w}_i\} \right].
    \end{equation*}

    The remaining closed set contains, in general, points where $\overline{u}_i$ is only Lipschitz-continuous. It
    is dealt with arguments on weak derivatives in Sobolev spaces, exactly as in the proof of the
    analogous result for the scalar comparison principle. This is quite standard and therefore not
    fully detailed here. We only consider the case where the intersection between the closed set
    and an open ball $B_0=B\left( \left( t_0,x_0 \right),r_0 \right)$ reduces to a portion
    of graph $\{ (t,x(t))\ |\ t\in(t_0-\tau,t_0+\tau) \}$ of class $\caC^1$ that divides the ball into two parts,
    the supergraph $B_0^+$ and the subgraph $B_0^-$. Noting that the first-order
    partial derivatives of $\overline{\veu}$ are well-defined in Lebesgue spaces, we have to 
    verify that, for any smooth scalar nonnegative test function $\varphi$ supported in $B_0$ and any $i\in[N]$,
    \begin{equation}\label{eq:generalized_supersolution}
	\int_{B_0}\partial_t\overline{u}_i\varphi
	+\int_{B_0}A_i\nabla\overline{u}_i\cdot\nabla\varphi
	+\int_{B_0}\left(q_i\cdot\nabla\overline{u}_i\right) \varphi
	-\int_{B_0}\left( \sum_{j=1}^N l_{i,j}\overline{u}_j+f_i(\overline{u}_i) \right)\varphi
	\geq 0
    \end{equation}

    By integration by parts,
    \begin{equation*}
	\int_{B_0^+}A_i\nabla\overline{u}_i\cdot\nabla\varphi = -\int_{B_0^+}\nabla\cdot(A_i\nabla\overline{u}_i)\varphi+\int_{\{(t,x(t)\}}\varphi A_i\nabla\overline{u}_i^+\cdot\nu^-,
    \end{equation*}
    \begin{equation*}
	\int_{B_0^-}A_i\nabla\overline{u}_i\cdot\nabla\varphi = -\int_{B_0^-}\nabla\cdot(A_i\nabla\overline{u}_i)\varphi+\int_{\{(t,x(t)\}}\varphi A_i\nabla\overline{u}_i^-\cdot\nu^+,
    \end{equation*}
    where $\nu^\pm$ denotes the normal to the graph of $t\mapsto x(t)$ pointing in $B_0^\pm$
    and $\overline{u}_i^\pm$ denotes the function of class $\caC^{1,2}$, either $\overline{v}_i$ or
    $\overline{w}_i$, that equals $\overline{u}_i$ in $B_0^\pm$.

    Hence, using the differential inequalities satisfied classically in the open sets 
    $\operatorname{int}(B_0^+)$ and $\operatorname{int}(B_0^-)$ as well as the relation 
    $\nu^+=-\nu^-$, \eqref{eq:generalized_supersolution} amounts to
    \begin{equation*}
	\int_{\{t,x(t)\}}\varphi A_i\nabla(\overline{u}_i^--\overline{u}_i^+)\cdot\nu^+\geq 0.
    \end{equation*}

    Up to reducing the radius $r_0$, there are two cases:
    \begin{itemize}
	\item either $i\in[N]$ is such that $\overline{u}_i^-=\overline{u}_i^+$ in $B_0$;
	\item or $i\in[N]$ is such that, in $B_0$, $\overline{u}_i^-$ and $\overline{u}_i^+$ coincide only on the graph 
	    $\{ (t,x(t))\ |\ t\in(t_0-\tau,t_0+\tau) \}$.
    \end{itemize}

    In the first case, $\nabla(\overline{u}_i^--\overline{u}_i^+)=0$ and \eqref{eq:generalized_supersolution} is obvious.

    In the second case, the graph of $t\mapsto x(t)$ can be understood as the $0$-level line of the function 
    $\overline{u}_i^--\overline{u}_i^+$. By regularity, we deduce that $\nabla(\overline{u}_i^--\nabla\overline{u}_i^+)\in\R\nu^+$
    in $\{(t,x(t))\}$. 
    From the definition of $\overline{u}_i$ as the minimum between $\overline{u}_i^-$ and $\overline{u}_i^+$, 
    we deduce more precisely that $\nabla(\overline{u}_i^--\overline{u}_i^+)\in\R^+\nu^+$ in $\{(t,x(t))\}$.
    The conclusion follows from the symmetry and positive definiteness of $A_i$ as well as the nonnegativity
    of the test function $\varphi$.
\end{proof}

We emphasize again that, due to the positivity of $\veC$ (\textit{cf.} \ref{ass:KPP}), the system \eqref{sys:KPP}
does not satisfy the assumptions of Proposition \ref{prop:comparison_principle}, and the global comparison principle turns 
out to be false
except in the special case $N=1$. Localized comparison principles, exploiting the fact that, at first-order, the reaction
term is cooperative around $\veu=\vez$ and competitive near $\veu=\vei$, were used in \cite{Cantrell_Cosner_Yu_2018}.
However the method in \cite{Cantrell_Cosner_Yu_2018} seems mostly limited to the special case $N=2$.
In this paper, we follow a different method, whose validity does not depend on $N$ at all.
}

\subsection{\red{Comparison between the components in present-time}}

\red{Here we introduce a nonlinear comparison between the components of the solution $\veu$ of \eqref{sys:KPP}--\eqref{IC}
away from the initial time. 

It relies upon Gaussian estimates for solutions of linear cooperative systems and more precisely on the following lemma, 
inspired by a work on nonlocal equations with a KPP structure \cite{Bouin_Henderso-1}.}

\begin{lem}\label{lem:comparison_of_components_in_real_time_linear}
    There exists $p\in(0,1)$ and $\kappa>0$ such that, for any solution $\vev$ of 
    the linear Cauchy problem
    \begin{equation*}
	\begin{cases}
	    \cbQ\vev=\vez & \text{in }(0,+\infty)\times\R^n, \\
	    \vev(0,\cdot)=\vev_{\upini} & \text{in }\R^n,
	\end{cases}
    \end{equation*}
    with $\vev_{\upini}\in\mathcal{L}^\infty(\R^n,[\vez,\vei))$,
    the following inequalities hold true:
    \begin{equation*}
	\forall i,j\in[N]\quad v_j(1,x)\leq \kappa \|\vev_{\upini}\|_{\mathcal{L}^\infty(\R^n,\R^N)}^{1-p} v_i(1,x)^p\quad\text{for all }x\in\R^n.
    \end{equation*}
\end{lem}
\begin{proof}
    Let $x\in\R^N$ be fixed.

    The proof uses two-sided Gaussian estimates at time $t=1$ on $\vect{\Gamma}_x$,
    the fundamental matrix solution of the linear Cauchy problem:
    \begin{equation*}
	\begin{cases}
	    \cbQ\vect{\Gamma}_x=\vez & \text{in }(0,+\infty)\times\R^n, \\
	    \vect{\Gamma}_x(0,y)=\delta_x(y)\vect{I} & \text{for all }y\in\R^n.
	\end{cases}
    \end{equation*}

    The upper Gaussian estimates have the following form: there exists constants $C_1,C_2>0$, independent of $x$, 
    such that
    \begin{equation}\label{eq:upper_gaussian_estimates}
	\Gamma_{x,i,j}(1,y)\leq C_1 \exp\left( -C_2|x-y|^2 \right)\quad\text{for all }i,j\in[N],\ y\in\R^n.
    \end{equation}
    We do not prove these quite standard upper Gaussian estimates and refer instead to \cite{Dong_Kim_2018} for the special
    case where each $q_i$ is divergence-free and to \cite[Chapter 9, Theorem 2]{Friedman_1983}, 
    \cite[Theorem 2.64]{Auscher_Egert_2022} for the general case. Note that such estimates do not require the 
    assumptions \ref{ass:cooperative}, \ref{ass:irreducible} on the structure of $\veL$; their proof only exploits a local 
    boundedness property of weak solutions.

    On the contrary, and as is well-known for scalar equations, lower Gaussian estimates are proved thanks to the Harnack 
    inequality of Proposition \ref{prop:harnack_inequality}. Since this Harnack inequality is derived from the assumptions 
    \ref{ass:cooperative}, \ref{ass:irreducible} and is therefore more specific to our setting, let us give a few details.
    By standard properties of the fundamental matrix solution, the column vector
    $\vew_j=(\Gamma_{x,i,j})_{i\in[N]}$, for any $j\in[N]$, is the solution of
    \begin{equation*}
	\begin{cases}
	    \cbQ\vew=\vez & \text{in }(0,+\infty)\times\R^n, \\
	    \vew(0,y)=\delta_x(y)\vect{e}_j & \text{for all }y\in\R^n.
	\end{cases}
    \end{equation*}
    From this observation, the comparison $\veL\geq\diag(l_{i,i})_{i\in[N]}$ and a standard duality argument, it is easily
    derived that $\Gamma_{x,j,j}\geq\Gamma_{x,j}$ in $(0,+\infty)\times\R^n$, where $\Gamma_{x,j}$ is the fundamental solution 
    associated with the scalar operator $\caP_j -l_{j,j}$ and the initial value $\delta_x$. Then, from lower Gaussian estimates 
    for fundamental solutions of scalar linear parabolic equations with space-time periodic coefficients \cite{Aronson_1968}, 
    it is deduced that there exists constants $c_1,c_2>0$ independent of $x$ such that, for all $j\in[N]$ and $y\in\R^n$,
    \begin{equation*}
	c_1 \exp\left( -c_2|x-y|^2 \right)\leq \Gamma_{x,j,j}(1/2,y).
    \end{equation*}
    By virtue of the Harnack inequality of Proposition \ref{prop:harnack_inequality}, there exists another
    constant $\kappa>0$ such that, for all $i,j\in[N]$ and $y\in\R^n$,
    \begin{equation*}
	\kappa\Gamma_{x,j,j}(1/2,y)\leq\Gamma_{x,i,j}(1,y)
    \end{equation*}
    so that, up to changing $c_1$, the following lower Gaussian estimates hold true:
    \begin{equation}\label{eq:lower_gaussian_estimates}
	c_1 \exp\left( -c_2|x-y|^2 \right)\leq \Gamma_{x,i,j}(1,y)\quad\text{for all }i,j\in[N],\ y\in\R^n.
    \end{equation}
    
    Next, by standard properties of the fundamental matrix solution,
    \begin{equation*}
	v_j(1,x) = \int_{\R^n}\left(\vect{\Gamma}_x(1,y)\vev_{\upini}(y)\right)_j\upd y.
    \end{equation*}
    Let $p'>\frac{c_2}{C_2}$ (note that \eqref{eq:upper_gaussian_estimates} and \eqref{eq:lower_gaussian_estimates}
    combined imply $p'>1$) and $s=\frac{c_2}{p'C_2}$ ($s\in(0,1)$), 
    so that $sC_2=\frac{c_2}{p'}$. Define $q'>1$ such that $\frac{1}{p'}+\frac{1}{q'}=1$.
    Using the upper Gaussian estimate \eqref{eq:upper_gaussian_estimates}
    and the H\"{o}lder inequality, we find:
    \begin{equation*}
	v_j(1,x)\leq C_1\left(\int_{\R^n} \upe^{-(1-s)q'C_2|x-y|^2}\upd y \right)^\frac{1}{q'}
	\left(\int_{\R^n}\upe^{-c_2|x-y|^2}\left(\sum_{k=1}^N v_{\upini}(y)\right)^{p'} \upd y \right)^\frac{1}{p'}
    \end{equation*}
    Hence, there exists a positive constant $C>0$, that depends only on $c_1$, $C_1$, $c_2$, $C_2$,
    and the choice of $p'$, such that:
    \begin{equation*}
	v_j(1,x)\leq C\left\|\sum_{k=1}^N v_{\upini}\right\|_{\mathcal{L}^\infty(\R^n)}^{\frac{p'-1}{p'}} \left(\int_{\R^n} c_1\upe^{-c_2|x-y|^2}\sum_{k=1}^N v_{\upini}(y) \upd y \right)^\frac{1}{p'}
    \end{equation*}

    Using the lower Gaussian estimate \eqref{eq:lower_gaussian_estimates}
    ends the proof with $p=1/p'$.
\end{proof}

\red{We deduce the following corollary, that will be used repeatedly and especially to control the solution $\veu$ of
\eqref{sys:KPP}--\eqref{IC} from below.

\begin{cor}\label{cor:comparison_of_components_in_real_time}
    Let $M>0$ such that 
    \begin{equation*}
	\max_{i\in[N]}\sup_{x\in\R^n}u_{\upini,i}(x)\leq M.
    \end{equation*}

    Then there exists $p\in(0,1)$ and $\kappa_M>0$ such that the solution $\veu$ of \eqref{sys:KPP}--\eqref{IC}
    satisfies
    \begin{equation*}
	\forall i,j\in[N]\quad u_j(t,x)\leq \kappa_M u_i(t,x)^p\quad\text{for all }(t,x)\in[1,+\infty)\times\R^n.
    \end{equation*}
\end{cor}}

\begin{proof}
    From the global bounds of Proposition \ref{prop:global_bound_absorbing_set}, we deduce the inequalities
\begin{equation*}
    \veL\veu-C\veu\leq\veL\veu-(\veC\veu)\circ\veu\leq\veL\veu\quad\text{in }[0,+\infty)\times\R^n
\end{equation*}
with
\begin{equation*}
    C=N(K+M)\max_{(i,j)\in[N]^2,(t,x)\in\clOmper}c_{i,j}(t,x)>0.
\end{equation*}
Let $t\geq 1$. By virtue of the comparison principle of Proposition \ref{prop:comparison_principle},
\begin{equation*}
    \upe^{-C}\vev(1,x)\leq\veu(t,x)\leq\vev(1,x)\quad\text{for all }x\in\R^n,
\end{equation*}
where $\vev$ solves
\begin{equation*}
    \begin{cases}
	\cbQ\vev=\vez & \text{in }(0,+\infty)\times\R^n, \\
	\vev(0,\cdot)=\veu(t-1,\cdot) & \text{in }\R^n.
    \end{cases}
\end{equation*}
Applying Lemma \ref{lem:comparison_of_components_in_real_time_linear} and Proposition
\ref{prop:global_bound_absorbing_set}, it follows that, 
for all $i,j\in[N]$ and $x\in\R^n$,
\begin{align*}
    u_j(t,x) & \leq v_j(1,x) \\
    & \leq \kappa\|\veu(t-1,\cdot)\|_{\mathcal{L}^\infty(\R^n,\R^N)}^{1-p} v_i(1,x)^p \\
    & \leq \kappa(K+M)^{1-p}\upe^{pC}u_i(t,x)^p.
\end{align*}

Setting $\kappa_M=\kappa(K+M)^{1-p}\upe^{pC}$ ends the proof.
\end{proof}

\section{Proofs}

Below we prove our results. In order to ease the reading, we first prove the extinction results
(Theorems \ref{thm:extinction} and \ref{thm:extinction_small_solutions}), that use super-solutions, 
then the persistence results (Theorems \ref{thm:existence_persistent_entire_solution}, 
\ref{thm:hair_trigger_effect} and \ref{thm:persistence_large_solutions}), that use sub-solutions. We conclude
with the proof of the Freidlin--G\"{a}rtner-type formula (Theorem \ref{thm:FG_formula})\red{, that uses
both super- and sub-solutions.}

\subsection{Global extinction (Theorem \ref{thm:extinction})}

It is convenient to distinguish two cases: $\lambda_1'>0$ on one hand, $\lambda_1'=0$ on the other 
hand. In the first case, the extinction is due to the linear part of the operator (and occurs 
therefore at an exponential rate). In the second case, however, the extinction is due to the 
signed quadratic part of the operator (it is conjectured to occur at an algebraic rate,
\textit{cf.} Remark \ref{rem:algebraic_extinction}). 
The following proofs are straightforward adaptations of \cite{Girardin_2016_2,Girardin_2016_2_add}.

\begin{proof}[Proof in the case $\lambda_1'>0$]
    The idea is very classical and consists in constructing a super-solution of the form
    \begin{equation*}
	\overline{\veu}:(t,x)\mapsto M\upe^{-\lambda_1' t}\veu_0(t,x),
    \end{equation*}
    where $\veu_0$ is a positive generalized principal eigenfunction associated with $\lambda_1'$ of fixed amplitude 
    and $M>0$ is a constant so large that $M\veu_0\gg\veu_{\upini}$. Then $\overline{\veu}-\veu$
    satisfies
    \begin{equation*}
	\cbQ(\overline{\veu}-\veu)=(\veC\veu)\circ\veu\geq\vez\quad\text{in }(0,+\infty)\times\R^n,
    \end{equation*}
    so that by the comparison principle of Proposition \ref{prop:comparison_principle},
    $\overline{\veu}\geq\veu$ globally in $(0,+\infty)\times\R^n$,
    and consequently $\veu$ vanishes asymptotically in time, uniformly in space, exponentially fast.
\end{proof}

\begin{proof}[Proof in the case $\lambda_1'=0$]
    This time we use a family of super-solutions of the form 
    \begin{equation*}
	\overline{\veu}_T:(t,x)\in(T,+\infty)\times\R^n\mapsto M_T\veu_0(t,x).
    \end{equation*}
    Assuming that $M_T>0$ is defined optimally for each $T\geq 0$, namely
    \begin{equation*}
	M_T=\sup_{x\in\R^n}\max_{i\in[N]}\frac{u_i(T,x)}{u_{0,i}(T,x)},
    \end{equation*}
    the goal is to prove that $M_T$ decreases to $0$ as $T\to+\infty$.

    By the comparison principle of Proposition \ref{prop:comparison_principle}, 
    $\overline{\veu}_T\geq\veu$ globally in $[T,+\infty)\times\R^n$.
    Therefore, for any $T'>T$, $M_{T'}\leq M_T$, simply by definition of $M_{T'}$. Hence
    the family $(M_T)_{T\geq 0}$ is nonincreasing.

    Of course, if $\veu_{\upini}=\vez$, then $M_0=0$, the family $(M_T)_{T\geq 0}$ is
    stationary and $\veu=\vez$, which ends the proof. From now on we discard this case and
    therefore assume $\veu_{\upini}\neq \vez$. Under such an assumption, let us prove
    that $(M_T)_{T\geq 0}$ is actually decreasing.

    Assume by contradiction that there exist $0\leq T<T'$ such that $M_T\leq M_{T'}$. Then, by
    large monotonicity, $t\mapsto M_t$ is constant in $[T,T']$. Below, we will begin by discarding
    the possibility that the optimum defining $M_T$ is attained pointwise, and subsequently we 
    will discard the possibility that it is attained asymptotically. 
    We recall the basis for our application
    of the strong comparison principle:
    \begin{equation*}
	\begin{cases}
	    \cbQ(\overline{\veu}_T-\veu)=(\veC\veu)\circ\veu\geq\vez & \text{in }[T,+\infty)\times\R^n, \\
	    \overline{\veu}_T(T,\cdot)-\veu(T,\cdot)\geq\vez & \text{in }\R^n.
	\end{cases}
    \end{equation*}

    If there exists $(t^\star,x^\star,i^\star)\in(T,T']\times\R^n\times[N]$ such that 
    $M_T u_{0,i^\star}(t^\star,x^\star)=u_{i^\star}(t^\star,x^\star)$,
    then by virtue of the strong comparison principle, $\veu=\overline{\veu}_T=M_T\veu_0$ globally in
    $[T,t^\star]\times\R^n$. \red{Therefore
    \begin{equation*}
	\cbQ\veu=M_T\cbQ\veu_0=M_T\lambda_{1,0}\veu_0=M_T\lambda_1'\veu_0=\vez\quad\text{in }[T,t^\star]\times\R^n,
    \end{equation*}
    where the assumption $\lambda_1'=0$ was used. But the definition of $\veu$ as the solution of the Cauchy problem
    \eqref{sys:KPP}--\eqref{IC} with $\veu_{\upini}\neq\vez$ implies $\veu\gg\vez$ in $(0,+\infty)\times\R^n$,
    which in turn implies $\cbQ\veu=-\veC\veu\circ\veu\ll\vez$.
    This is a contradiction.} Hence there does not exist such a triplet $(t^\star,x^\star,i^\star)$.

    Consequently, for all $(t,x,i)\in(T,T']\times\R^n\times[N]$, 
    $M_T u_{0,i}(t,x)>u_i(t,x)$ and the equality can only be attained asymptotically at $|x|=\infty$.

    Fix temporarily $t_0\in(T,T')$ and let $i\in[N]$ such that 
    \begin{equation*}
	M_T=\sup_{x\in\R^n}\frac{u_i(t_0,x)}{u_{0,i}(t_0,x)}.
    \end{equation*}
    There exists a sequence $(x_k)\in(\R^n)^{\N}$ such that $|x_k|\to+\infty$ and 
    \begin{equation*}
	\frac{u_i(t_0,x_k)}{u_{0,i}(t_0,x_k)}\to M_T\quad\text{as }k\to+\infty.
    \end{equation*}
    We intend to use the spatial periodicity and therefore we define, for each $k\in\N$,
    $y_k\in[0,L]$ and $z_k\in\prod_{\alpha\in[n]}L_\alpha\Z$ such that $x_k=y_k+z_k$. 
    Up to extraction, the sequence $(y_k)$ converges to a limit $y_\infty\in[0,L]$.
    Then, by classical parabolic estimates \cite{Lieberman_2005} and up to a diagonal extraction, the sequence $(\veu_k)$ defined by
    \begin{equation*}
	\veu^k:(t,x)\mapsto\veu(t,x+z_k)\quad\text{for each }k\in\N
    \end{equation*}
    converges locally uniformly to a solution $\veu^\infty$ of \eqref{sys:KPP}. The solution $\veu^\infty$ satisfies moreover 
    $M_T u_{0,i}(t_0,y_\infty)=u^\infty_{i}(t_0,y_\infty)$ and also 
    $\veu^\infty\leq M_T\veu_0$ globally in $(T,T')\times\R^n$.
    Repeating the previous strong comparison argument, we deduce again a contradiction.

    Therefore $(M_T)_{T\geq 0}$ is decreasing and converges to a limit $M_\infty\geq 0$.
    Assume by contradiction $M_\infty>0$. Then there exist $i\in[N]$ and a sequence $(t_k,x_k)$ 
    such that 
    \begin{equation*}
	t_k\to+\infty,\quad |x_k|\to+\infty,\quad \frac{u_i(t_k,x_k)}{u_{0,i}(t_k,x_k)}\to M_\infty\quad\text{as }k\to+\infty.
    \end{equation*}
    Arguing exactly as before after passing to the limit $k\to+\infty$ locally uniformly, a new 
    contradiction arises. In the end, $M_\infty=0$, and by definition of $M_T$, $\veu$
    vanishes asymptotically in time, uniformly in space.
\end{proof}

\begin{rem}\label{rem:algebraic_extinction}
    We are unable to prove the algebraic decay in the case $\lambda_1'=0$. It is strongly conjectured
    in view of the quadratic nonlinearity (by analogy with the ODE $u'=-u^2$), but it remains 
    as an open question. Some technical obstacles are discussed in 
    \cite[Section 4.1.1]{Girardin_2016_2}.
\end{rem}

\subsection{Conditional extinction of small solutions (Theorem \ref{thm:extinction_small_solutions})}
In this section, we prove that if there exists $z\in\R^n$ such that:
\begin{enumerate}
    \item $\lambda_{1,z}\geq 0$;
    \item there exists $C>0$ such that, for all $x\in\R^n$, $\max_{i\in[N]}u_{\upini,i}(x)\leq C\upe_z(x)$;
\end{enumerate}
then the solution of the Cauchy problem goes extinct locally uniformly.

(Note that $\lambda_{1,z}\geq 0$ implies $\lambda_1\geq 0$.)

The proof is very similar to that of Theorem \ref{thm:extinction} but applies the comparison principle to a 
semilinear system instead of a linear one, in order to account for the unboundedness of the eigenfunctions.

\begin{proof}
    \red{We apply Proposition \ref{prop:global_bound_absorbing_set} to obtain a global bound $\veu\leq(K+M)\veo$. 

    Let $T\geq 0$. We define, for some $M_T>0$ to be made precise in a moment,
    \begin{equation*}
	\overline{\vev}:(t,x)\in[T,+\infty)\times\R^n\mapsto M_T\upe^{-\lambda_{1,z}t}\upe^{z\cdot x}\veu_{z}(t,x),
    \end{equation*}
    \begin{equation*}
	\overline{\veu}:(t,x)\in[T,+\infty)\times\R^n\mapsto \left(\min\left(\overline{v}_i(t,x),K+M\right)\right)_{i\in[N]},
    \end{equation*}
    where $\veu_z\in\caC^{1,2}_{\upp}(\R\times\R^n,(\vez,\vei))$ is a positive periodic principal eigenfunction of the operator
    $\upe_{-z}\cbQ\upe_z$.

    On one hand, by definition, $\overline{\vev}(t,x)$ satisfies
    \begin{equation*}
	\cbQ\overline{\vev}=-\lambda_{1,z}\overline{\vev}+\lambda_{1,z}\overline{\vev}=\vez\quad\text{in }(T,+\infty)\times\R^n
    \end{equation*}
    so that
    \begin{equation*}
	\cbQ\overline{\vev}\geq -\diag(c_{i,i})\overline{\vev}\circ\overline{\vev}\quad\text{in }(T,+\infty)\times\R^n.
    \end{equation*}

    On the other hand, in view of the construction of $K$, we can assume without loss of generality that
    \begin{equation*}
	\veL\left( \left( K+M \right)\veo \right)-\diag(c_{i,i})\left( \left( K+M \right)\veo \right)\circ\left( \left( K+M \right)\veo \right)\leq\vez\quad\text{in }(T,+\infty)\times\R^n.
    \end{equation*}

    Hence by virtue of Proposition \ref{prop:generalized_comparison_principle} we can use $\overline{\veu}$ as 
    a bounded super-solution to be compared with the sub-solution $\veu$:
    \begin{equation*}
	\cbQ\veu=-\veC\veu\circ\veu\leq-\diag(c_{i,i})\veu\circ\veu\quad\text{in }(T,+\infty)\times\R^n.
    \end{equation*}

    In the case $T=0$, we choose $M_0$ appropriately large so that $\overline{\veu}(0,\cdot)\geq\veu_{\upini}$
    and we deduce from the comparison principle that $\overline{\veu}\geq\veu$ globally in $[0,+\infty)\times\R^n$.

    In the case $T>0$, we use the exponential bound on $\veu$ deduced from the case $T=0$ to show the existence
    of $M_T<+\infty$ such that $\overline{\veu}(T,\cdot)\geq\veu(T,\cdot)$, and we deduce similarly that $\overline{\veu}\geq\veu$
    globally in $[T,+\infty)\times\R^n$.}

    To conclude, as in the proof of Theorem \ref{thm:extinction}, we distinguish two cases:
    \begin{itemize}
	\item in the case $\lambda_{1,z}>0$, the super-solution with $T=0$ vanishes locally uniformly 
	    so that the solution also vanishes locally uniformly;
	\item in the case $\lambda_{1,z}=0$, as in the case $\lambda_1'=0$ of the proof of Theorem \ref{thm:extinction},
	    we assume $M_T$ to be optimal,
	    \begin{equation*}
		M_T=\sup_{x\in\R^n}\max_{i\in[N]}\frac{u_i(T,x)}{\upe^{z\cdot x}u_{z,i}(T,x)},
	    \end{equation*}
	    and show by comparison and limiting arguments that as $T\to+\infty$ it decreases to zero if $\veu_{\upini}$ 
	    is nonzero. 
    \end{itemize}
\end{proof}

\subsection{Existence of a nonnegative nonzero space-time periodic entire solution (Theorem \ref{thm:existence_persistent_entire_solution})}

In order to prove Theorem \ref{thm:existence_persistent_entire_solution}, we adapt the 
arguments of the proof of \cite[Theorem 2.3]{Alfaro_Griette}, which is a similar result but
established under more restrictive assumptions ($N=2$, the coefficients are space-periodic but 
time-homogeneous) and with more precise conclusions (the constructed solutions are space-periodic 
time-independent stationary states).

The proof involves the following bifurcation theorem \cite[Theorem 3.1]{Alfaro_Griette}, that
we recall for clarity. The notations in the following statement are completely independent
from the notations in the rest of the paper.

\begin{thm}\label{thm:Alfaro_Griette}
    Let $E$ be a Banach space and $C\subset E$ be a closed convex cone with nonempty interior
    and vertex $0$ (\textit{i.e.}, $C\cap-C=\{0\}$). Let $F:\R\times E\to E$ be a continuous
    compact operator and
    \begin{equation*}
	\mathcal{S}=\overline{\left\{ (\alpha,x)\in\R\times E\backslash\{0\}\ |\ F(\alpha,x)=x \right\}}
    \end{equation*}
    and 
    \begin{equation*}
	\mathbb{P}_{\R}\mathcal{S}=\left\{ \alpha\in\R\ |\ \exists x\in C\backslash\{0\}\ (\alpha,x)\in\mathcal{S} \right\}.
    \end{equation*}

    Assume the following properties.
    \begin{enumerate}
	\item For all $\alpha\in\R$, $F(\alpha,0)=0$.
	\item $F$ is Fréchet differentiable near $\R\times\{0\}$ with derivative $\alpha T$
	    locally uniformly with respect to $\alpha$.
	\item $T$ is strongly positive in the sense of the Krein--Rutman theorem: 
	    $T(C\backslash\{0\})\subset\operatorname{int}(C)$. Its Krein--Rutman eigenvalue
	    is denoted $\rho(T)>0$.
	\item $\mathcal{S}\cap(\{\alpha\}\times C)$ is bounded locally uniformly with respect to 
	    $\alpha\in\R$.
	\item $\mathcal{S}\cap(\R\times(\partial C\backslash\{0\}))=\emptyset$.
    \end{enumerate}

    Then, either $\left( -\infty,\frac{1}{\rho(T)} \right)\subset\mathbb{P}_{\R}\mathcal{S}$
    or $\left( \frac{1}{\rho(T)},+\infty \right)\subset\mathbb{P}_{\R}\mathcal{S}$.
\end{thm}

In our case, the Banach space will be $\mathcal{C}^{1+\delta/2,2+\delta}_{\upp}(\R\times\R^n,\R^N)$,
the cone will be $\mathcal{C}^{1+\delta/2,2+\delta}_{\upp}(\R\times\R^n,[\vez,\vei))$
\footnote{In \cite{Alfaro_Griette}, the authors chose the Banach space 
$L^\infty_{\upp}(\R\times\R^n,\R^N)$. However the cone of nonnegative functions in this space
has empty interior. More generally, it is known that the Krein--Rutman theorem cannot be 
directly applied in $L^p$ spaces. This is why usually a space of H\"{o}lder-continuous 
functions is used instead. Indeed, replacing $L^\infty$ by $\mathcal{C}^{2+\delta}$
all along the proof in \cite{Alfaro_Griette} corrects it without further issues.},
and $F$ will be the mapping $(\alpha,\vect{f})\mapsto\veu$ where
$\veu$ is the space-time periodic solution of 
$\cbQ\veu+M\veu=-\veC\vect{f}\circ\vect{f}+\alpha\vect{f}$. In other words, 
$F(\alpha,\vect{f})=(\cbQ+M)^{-1}(-\veC\vect{f}\circ\vect{f}+\alpha\vect{f})$.
The invertibility of $\cbQ+M$ is obviously false for some values of $M\in\R$, but is true
once $M>0$ is large enough, and this is why this parameter is introduced.

The derivative $T$ at $\R\times\{\vez\}$ can be easily identified as $\alpha(\cbQ+M)^{-1}$ 
(\textit{cf.} \cite{Alfaro_Griette}). 
Note that the Krein--Rutman eigenvalue $\rho(T)$ of $T$ is related 
to the generalized principal eigenvalue $\lambda_1'$ of $\cbQ$ via the relation 
$1/\rho(T)=\lambda_1'+M$. 

Therefore, keeping in mind that, when $\cbQ$ is replaced by 
$\cbQ+M-\alpha$, the extinction case of Theorem \ref{thm:extinction} corresponds to
$0\leq\lambda_1'+M-\alpha$, 
\textit{i.e.} for all $\alpha\leq\lambda_1'+M$, the conclusion of Theorem \ref{thm:Alfaro_Griette}
will read as $\left( \lambda_1'+M,+\infty \right)\subset\mathbb{P}_{\R}\mathcal{S}$.
By continuity, any $(\alpha,\veu)\in\mathcal{S}$ satisfies $F(\alpha,\veu)=\veu$. 
Therefore Theorem \ref{thm:Alfaro_Griette} brings forth a nonnegative nonzero solution 
$\veu$ of $\cbQ\veu+M\veu=-\veC\veu\circ\veu+\alpha\veu$ for any $\alpha>\lambda_1'+M$.

In order to conclude the proof, it will suffice then to observe that:
\begin{enumerate}
    \item $M>\lambda_1'+M$;
    \item all nonnegative nonzero solutions $\veu$ of $\cbQ\veu=-\veC\veu\circ\veu$
	are actually uniformly positive (as a particular case of assumption 5 in
	Theorem \ref{thm:Alfaro_Griette}).
\end{enumerate}

In view of this sketch of the proof, it only remains to verify the assumptions of 
Theorem \ref{thm:Alfaro_Griette}. The first one is immediate; the second one is just 
standard calculus \cite{Alfaro_Griette}; the third one is a standard consequence of 
the comparison principle for cooperative fully coupled systems once $M$ is chosen sufficiently 
large (\textit{e.g.}, \cite{Bai_He_2020,Girardin_Mazari_2022}).
Assumption 4 (the fixed points of $F(\alpha,\cdot)$ are locally bounded) is satisfied
by virtue of Corollary \ref{cor:global_bound_entire} (the parameters $\alpha$ and $M$ change
nothing to the argument and the obtained bound is indeed locally uniform with respect to 
$\alpha$).
It only remains to verify assumption 5 (there is no fixed point of $F(\alpha,\cdot)$ on the 
boundary of the cone). This is the object of the following lemma. For ease of reading, we
only state the case $\alpha=M$, without loss of generality.

\begin{lem}
    All nonnegative nonzero space-time periodic solutions $\veu$ of \eqref{sys:KPP} satisfy:
    \begin{equation*}
	\min_{(t,x)\in\clOmper}\min_{i\in[N]}u_i(t,x)>0.
    \end{equation*}
\end{lem}

\begin{proof}
    We rewrite the semilinear operator $\veu\mapsto\cbQ\veu+\veC\veu\circ\veu$ as a space-time heterogeneous
    linear operator $\cbQ+\diag(\veC\veu)$. Provided $\veu$ is a space-time periodic classical solution, 
    this linear operator has space-time periodic globally bounded coefficients and is cooperative and
    fully coupled, and therefore we are in a position to apply the strong comparison principle. The 
    conclusion follows directly.
\end{proof}

\subsection{Hair-trigger effect (Theorem \ref{thm:hair_trigger_effect})}

\red{This proof exploits the present-time comparison of Corollary \ref{cor:comparison_of_components_in_real_time}
and the limit $\lambda_1=\lim_{R\to+\infty}\lambda_{1,\upDir}(B(0,R))$ of
time-periodic Dirichlet principal eigenvalues in balls of increasing radius.

We intend to prove that, for every $R>0$,
\begin{equation*}
    \liminf_{t\to+\infty}\min_{i\in[N]}\min_{|x|\leq R}u_i(t,x)>0.
\end{equation*}

Let $M>0$ such that $\max_{i\in[N]}\sup_{x\in\R^n}u_{\upini,i}(x)\leq M$.
Thanks to Corollary \ref{cor:comparison_of_components_in_real_time}, there exists $p\in(0,1)$ and
$\kappa_M>0$ such that, at any time $t\geq 1$ and for each $i\in[N]$,
\begin{equation*}
    \left(\veC\veu\right)_{i\in[N]}(t,\cdot)\leq\sum_{j=1}^N (c_{i,j}\kappa_M u_i^p)(t,\cdot)
    \leq\kappa_M\left(\sum_{j=1}^N\max_{(t,x)\in\clOmper}c_{i,j}\right) u_i^p(t,\cdot).
\end{equation*}
Denoting 
\begin{equation*}
    D=\kappa_M\max_{i\in[N]}\left(\sum_{j=1}^N\max_{(t,x)\in\clOmper}c_{i,j}\right),
\end{equation*}
it follows that, at any time $t\geq 1$,
\begin{equation*}
    (\veL\veu-\veC\veu\circ\veu)(t,\cdot)\geq (\veL\veu-D\veu^{\circ(1+p)})(t,\cdot)\quad\text{in }\R^n.
\end{equation*}
The right-hand side defines a new semilinear reaction term which is cooperative, contrarily to the
original KPP reaction term. By the comparison principle of Proposition \ref{prop:comparison_principle}, any solution of 
$\cbQ\underline{\veu}=-D\underline{\veu}^{\circ(1+p)}$ with
$\vez\leq\underline{\veu}(1,\cdot)\leq\veu(1,\cdot)$ will satisfy $\vez\leq\underline{\veu}\leq\veu$ 
globally in space at any time $t\geq 1$.

Therefore, easing the notations and shifting the initial time, it only remains to prove that
any solution $\vev$ of $\cbQ\vev+D\vev^{\circ(1+p)}=\vez$ with nonnegative nonzero initial 
condition $\vev_{\upini}$ satisfies, for every $R>0$,
\begin{equation}\label{eq:Cauchy_persistence_cooperative_system}
    \liminf_{t\to+\infty}\min_{i\in[N]}\min_{|x|\leq R}v_i(t,x)>0.
\end{equation}
Since it suffices to prove \eqref{eq:Cauchy_persistence_cooperative_system} for large values of $R$, we fix
with no loss of generality a value of $R$ so large that $\lambda_{1,\upDir}(B(0,R+1))<0$. This choice is made possible
by the assumption $\lambda_1<0$.

\begin{proof}[Proof of the hair-trigger effect for the cooperative semilinear system]
    At time $t=1$, the solution $\vev$ of 
    \begin{equation*}
	\begin{cases}
	    \cbQ\vev+D\vev^{\circ(1+p)}=\vez & \text{in }(0,+\infty)\times\R^n, \\
	    \vev(\vez,\cdot)=\vev_{\upini} & \text{in }\R^n,
	\end{cases}
    \end{equation*}
    satisfies $\vev(1,x)\gg\vez$ for all $x\in\R^n$. In particular, by continuity,
    there exists $\varepsilon>0$ such that $\vev(1,\cdot)\geq\varepsilon\veo$ in $\overline{B(0,R+1)}$.

    Let $\lambda=\lambda_{1,\upDir}(B(0,R+1))$.
    We define $\underline{\vev}:(t,x)\mapsto\varepsilon\veu_{\upDir}(t,x)$ 
    where $\veu_{\upDir}$ solves in $\R\times\overline{B(0,R+1)}$:
    \begin{equation*}
	\begin{cases}
	    \cbQ\veu_{\upDir}=\lambda\veu_{\upDir} & \text{in }\R\times B(0,R+1), \\
	    \veu_{\upDir}=\vez & \text{on }\R\times\partial B(0,R+1), \\
	    \veu_{\upDir}\text{ time-periodic}, \\
	    \displaystyle\max_{(t,x)\in[0,T]\times \overline{B(0,R+1)}}\max_{i\in[N]}u_{\upDir,i}(t,x)=1.
	\end{cases}
    \end{equation*}
    and is extended to $\R\times\R^n$ by setting $\veu_{\upDir}(\cdot,x)=\vez$ if $|x|>R+1$.

    By construction, $\underline{\vev}(1,\cdot)\leq\vev(1,\cdot)$ in $B(0,R+1)$, and $\cbQ\underline{\vev}=\lambda\underline{\vev}$
    in $(1,+\infty)\times B(0,R+1)$.
    Up to reducing $\varepsilon$,
    \begin{equation*}
	-\lambda\geq D\varepsilon^p\quad\text{in }\R\times B(0,R+1),
    \end{equation*}
    whence, by virtue of $\vez\leq\veu_{\upDir}\leq\veo$,
    \begin{equation*}
	\cbQ\underline{\vev}\leq -D\underline{\vev}^{\circ(1+p)}\quad\text{in }\R\times B(0,R+1).
    \end{equation*}
    Therefore, by applying the comparison principle (not directly Proposition \ref{prop:comparison_principle}, but a similar
    statement for problems in a smooth bounded domain, larger than the spatial periodicity cell, with Dirichlet boundary 
    conditions, \textit{cf.} \cite[Proposition 2.2]{Girardin_Mazari_2022}),
    we deduce that $\vev\geq\underline{\vev}$ in $[1,+\infty)\times B(0,R+1)$.

    The continuity of $\veu_{\upDir}$ and its positivity in $\R\times\overline{B(0,R)}\subset \R\times B(0,R+1)$ end the proof of
    \eqref{eq:Cauchy_persistence_cooperative_system}.
\end{proof}

\begin{rem}
    Another proof following the ideas of our previous paper \cite[Theorem 1.3]{Girardin_2016_2} and relying upon the Harnack
    inequality of Proposition \ref{prop:harnack_inequality} instead of the present-time comparison of 
    Corollary \ref{cor:comparison_of_components_in_real_time} is possible. It is much more lengthy and technical.
    In this paper we favor the shortest proof, made possible by the novel Corollary 
    \ref{cor:comparison_of_components_in_real_time}.
\end{rem}
}

\red{With similar arguments and an additional optimization of $\varepsilon$,
we can obtain with the same method a uniform positivity bound for entire solutions, stated below without proof.}

\begin{prop}
    Assume $\lambda_1<0$. Then there exists a constant $\nu>0$ such that all nonnegative nonzero globally bounded 
    entire solutions $\veu$ of \eqref{sys:KPP} satisfying 
    \begin{equation*}
	\min_{i\in[N]}\inf_{(t,x)\in\R\times\R^n}u_i(t,x)>0
    \end{equation*}
    actually satisfy
    \begin{equation*}
	\veu\geq\nu\veo\quad\text{in }\R\times\R^n.
    \end{equation*}
\end{prop}

\red{Subsequently, a diagonal extraction and a limiting argument show that the same uniform bound applies, asymptotically,
to solutions of the Cauchy problem.}

\begin{cor}
    Assume $\lambda_1<0$. Then all solutions $\veu$ of
    the Cauchy problem \eqref{sys:KPP}--\eqref{IC} with nonzero initial conditions $\veu_{\upini}$ satisfy
    \begin{equation*}
	\liminf_{t\to+\infty}\min_{i\in[N]}\min_{|x|\leq R}u_i(t,x)\geq\nu\quad\text{for all }R>0.
    \end{equation*}
\end{cor}

\subsection{Conditional persistence of large solutions (Theorem \ref{thm:persistence_large_solutions})}
\label{sec:proof_persistence_large_solutions}
Let us recall the assumptions of Theorem \ref{thm:persistence_large_solutions}. There exists $z\in\R^n$ such that:
\begin{enumerate}
    \item $\lambda_{1,z}<0$;
    \item $\zeta\in(0,2)\mapsto\lambda_{1,\zeta z}$ is increasing in a neighborhood of $1$;
    \item there exists $C>0$, $B\in\R$ such that, for all $x\in\R^n$ such that $z\cdot x\leq B$, 
        $\min_{i\in[N]}u_{\upini,i}(x)\geq C^{-1}\upe_{z}(x)$.
\end{enumerate}

\red{Just as in the proof of Theorem \ref{thm:hair_trigger_effect},} by Corollary \ref{cor:comparison_of_components_in_real_time}, 
there exists $D>0$ and $p\in(0,1)$ such that, starting from $t=1$,
\begin{equation*}
    \veL\veu-(\veC\veu)\circ\veu\geq \veL\veu-D\veu^{\circ(1+p)}.
\end{equation*}
By virtue of the comparison principle \red{of Proposition \ref{prop:comparison_principle}}, any solution of 
$\cbQ\underline{\veu}=-D\underline{\veu}^{\circ(1+p)}$ with
$\vez\leq\underline{\veu}(1,\cdot)\leq\veu(1,\cdot)$ will satisfy $\underline{\veu}\leq\veu$ globally in space at any time $t\geq 1$.
Therefore it only remains to prove that:
\begin{enumerate}
    \item $\min_{i\in[N]}u_i(1,x)\geq \widetilde{C}^{-1}\upe_z(x)$ in $\{z\cdot x\leq\widetilde{B}\}$ for some $\widetilde{C}>0$, $\widetilde{B}\in\R$;
    \item the persistence result is true for the cooperative system 
	$\cbQ\vev=-D\vev^{\circ(1+p)}$.
\end{enumerate}

\begin{proof}[Proof of the exponential estimates at $t=1$]
    There exists a large $E>0$ such that each $u_i$ admits a (rough) sub-solution $v_i\leq u_i$ satisfying
    \begin{equation*}
        \begin{cases}
    	\caP_i v_i = -E v_i & \text{in }(0,+\infty)\times\R^n \\
    	v_i(0,x)= \frac{1}{C}\upe_z(x) & \text{in }\{z\cdot x\leq B\} \\
    	v_i(0,x)=0 & \text{in }\{z\cdot x>B\}.
        \end{cases}
    \end{equation*}
    Such sub-solutions can be related to solutions of
    \begin{equation*}
        \begin{cases}
    	(\upe_{-z}\caP_i)(\upe_z w_i) = -E w_i & \text{in }(0,+\infty)\times\R^n \\
    	w_i(0,x)= \frac{1}{C} & \text{in }\{z\cdot x\leq B\} \\
    	w_i(0,x)=0 & \text{in }\{z\cdot x>B\}
        \end{cases}
    \end{equation*}
    through the formula $v_i=\upe_z w_i$. Hence we only have to show that, for each $i\in[N]$, $\liminf_{z\cdot x\to-\infty}w_i(1,x)>0$.
    Recall that the operator $\upe_{-z}\caP_i\upe_z$ has the following form:
    \begin{equation*}
        \upe_{-z}\caP_i\upe_z = \caP_i-\left(2A_i z\cdot\nabla+z\cdot A_i z+\nabla\cdot\left(A_i z\right)-q_i\cdot z\right).
    \end{equation*}
    Let $(x_k)_{k\in\N}$ such that $z\cdot x_k\to-\infty$ and such that $w_i(1,x_k)\to\liminf_{z\cdot x\to-\infty} w_i(1,x)$. 
    By classical parabolic estimates \cite{Lieberman_2005} and a diagonal extraction, up to a subsequence, the sequence of 
    functions $(t,x)\mapsto w_i(t,x-x_n)$ converges locally uniformly to the solution $w_i^\infty$ of
    \begin{equation*}
        \begin{cases}
    	(\upe_{-z}\caP_i)(\upe_z w_i^\infty) = -E w_i^\infty & \text{in }(0,+\infty)\times\R^n \\
    	w_i^\infty(0,\cdot)= \frac{1}{C} & \text{in }\R^n
        \end{cases}
    \end{equation*}
    By constructing yet another sub-solution that solves a mere ODE of the form $w'=-Fw$, we deduce easily that $w_i^\infty(1,0)>0$.
    This ends the proof of this step.
\end{proof}

The last part of the proof is its core and is a straightforward adaptation of the proof in the scalar case
\cite{Nadin_2010}. Easing the notations, we consider the cooperative system $\cbQ\vev=-D\vev^{\circ(1+p)}$ with
nonnegative globally bounded initial conditions $\vev_{\upini}$ satisfying $\min_{i\in[N]}v_{\upini,i}(x)\geq C^{-1}\upe_z(x)$ 
in $\{z\cdot x\leq B\}$. For this system, let us prove that, for any fixed $R>0$,
\begin{equation*}
    \liminf_{t\to+\infty}\min_{i\in[N]}\min_{|x|\leq R}v_i(t,x)>0.
\end{equation*}

\begin{proof}[Proof of the persistence result for the cooperative semilinear system]
    Let $\zeta\in(0,1)$ such that 
    \begin{equation*}
	\lambda_{1,z}<\lambda_{1,(1+\zeta)z}<0.
    \end{equation*}
    We define
    \begin{equation*}
	\underline{\veu}=\frac{1}{C}\upe_z\veu_z-A\upe_{(1+\zeta)z}\veu_{(1+\zeta)z}
    \end{equation*}
    and our goal is to use $[\underline{\veu}]^+$ as a (generalized) sub-solution in 
    Proposition \ref{prop:generalized_comparison_principle},
    where the notation $[\cdot]^+$ refers to the component-by-component positive part of a vector.

    On one hand, whatever $A>0$ is, $\underline{\veu}(0,\cdot)\leq\vev_{\upini}$ in $\{z\cdot x\leq B\}$.
    On the other hand, $A>0$ can be chosen so large that $\underline{\veu}(0,\cdot)\leq\vev_{\upini}$ in $\{z\cdot x>B\}$.
    Hence, with such an appropriate choice of $A$, 
    \begin{equation*}
	[\underline{\veu}(0,\cdot)]^+\leq\vev_{\upini}\quad\text{globally in }\R^n.
    \end{equation*}

    \red{In order to apply the comparison principle of Proposition \ref{prop:generalized_comparison_principle} to
    the generalized sub-solution $[\underline{\veu}]^+$, it only remains to verify that 
    $\cbQ\vez\leq-D\vez^{\circ(1+p)}$ and $\cbQ\underline{\veu}\leq-D\underline{\veu}^{\circ(1+p)}$ hold true
    globally in $(0,+\infty)\times\R^n$. Note that this requires an extension of $\vev\mapsto\vev^{\circ(1+p)}$
    outside $[\vez,\vei)$, of class $\caC^1$. Hence we set $v^{1+p}=-|v|^{1+p}=v|v|^p$ if $v<0$, and consistently
    $\vev^{\circ(1+p)}=(v_i^{1+p})_{i\in[N]}$.

    The inequality $\cbQ\vez\leq-D\vez^{\circ(1+p)}$ is obvious.}

    Let $\varepsilon\in(0,|\lambda_{1,z}|)$. By increasing $A$ if necessary, 
    \begin{equation*}
	D\left(\left[\underline{\veu}\right]^+\right)^{\circ p}\leq\varepsilon\veo.
    \end{equation*}
    Therefore
    \begin{align*}
	\cbQ\underline{\veu} & =\frac{1}{C}\lambda_{1,z}\upe_z\veu_z -A\lambda_{1,(1+\zeta)z}\upe_{(1+\zeta)z}\veu_{(1+\zeta)z} \\
	& \leq -\varepsilon\underline{\veu}+\varepsilon\underline{\veu}+\lambda_{1,z}\left( \frac{1}{C}\upe_z\veu_z-A\upe_{(1+\zeta)z}\veu_{(1+\zeta)z} \right) \\
	& \leq -D\left(\left[\underline{\veu}\right]^+\right)^{\circ p}\underline{\veu}+(\varepsilon+\lambda_{1,z})\underline{\veu} \\
	& \leq -D\underline{\veu}^{\circ(1+p)}.
    \end{align*}

    Hence, by virtue of the comparison principle applied to the semilinear
    cooperative operator $\cbQ\vev+D\vev^{\circ(1+p)}$, the inequality $\vev\geq\left[\underline{\veu}\right]^+$ 
    is satisfied globally in $[0,+\infty)\times\R^n$.

    By the special form of $\underline{\veu}$, there exists $x_0\in\R^n$ such that 
    \begin{equation*}
	\min_{i\in[N]}\inf_{t\geq 0}\underline{u}_i(t,x_0)>0. 
    \end{equation*}
    Consequently, 
    \begin{equation*}
	\min_{i\in[N]}\inf_{t\geq 0}v_i(t,x_0)>0. 
    \end{equation*}
    Now, up to increasing without loss of generality $R$ so that $x_0\in B(0,R)$,
    the Harnack inequality of Proposition \ref{prop:harnack_inequality} yields the existence of a constant $\kappa>0$ 
    such that, for all $t\geq 1$,
    \begin{equation*}
	\min_{i\in[N]}\min_{x\in \overline{B(0,R)}}v_i(t+1,x)
	\geq \kappa\max_{i\in[N]}\max_{x\in \overline{B(0,R)}}v_i(t,x)
    \end{equation*}
    Subsequently,
    \begin{equation*}
	\min_{i\in[N]}\min_{x\in \overline{B(0,R)}}v_i(t+1,x)
	\geq \kappa\min_{i\in[N]}v_i(t,x_0)
	\geq \kappa\min_{i\in[N]}\inf_{t\geq 0}v_i(t,x_0)>0.
    \end{equation*}
    This ends the proof.
\end{proof}

\subsection[The Freidlin--G\"{a}rtner-type formula]{The Freidlin--G\"{a}rtner-type formula (Theorem \ref{thm:FG_formula})}

In this section, we assume that $\lambda_1<0$ and $\veu_{\upini}$ is compactly supported, and we prove
the spreading speed estimates \eqref{eq:spreading_subestimate} and \eqref{eq:spreading_superestimate}.

To this end, we fix once and for all $e\in\mathbb{S}^{n-1}$.

First, we confirm that the minima involved in the definitions of $c^\star_e$ and $c^{\upFG}_e$ are
indeed well-defined.

\begin{lem}\label{lem:minimal_c}
    \red{The infimum $c^\star_e\in[0,+\infty)$ of the set $\{ c^\mu_e\ |\ \mu>0\}$ is a minimum attained at a unique $\mu^\star>0$.

    Consequently, $c^\star_e>0$.}

    Moreover, for any $c>c^\star_e$, there exist $\mu_1,\mu_2>0$ such that $\mu_1<\mu_2$ and $c^{\mu_1}_e=c^{\mu_2}_e=c$,
    whereas for any $c\in(0,c^\star_e)$, there exists no $\mu>0$ such that $c^\mu_e=c$.
\end{lem}

\begin{proof}
    We consider the function $\psi:\mu\in[0,+\infty)\mapsto-\lambda_{1,-\mu e}-c\mu$.
    By continuity and strict concavity of $\mu\mapsto\lambda_{1,-\mu e}$ 
    \cite[Corollary 3.6]{Girardin_Mazari_2022}
    and by virtue of the quadratic growth of $|\lambda_{1,-\mu e}|$ as $\mu\to +\infty$
    \cite[Corollary 3.13]{Girardin_Mazari_2022}),
    $\psi$ is continuous, strictly convex, and satisfies
    \begin{equation*}
	\psi(0)=-\lambda_1'\geq-\lambda_1>0,\quad\lim_{\mu\to +\infty}\psi(\mu)=+\infty.
    \end{equation*}
    It admits a global minimum in $[0,+\infty)$. Since
    \begin{equation*}
	\psi(0)=-\lambda_1',\quad\psi(1)=-\lambda_{1,-e}-c,
    \end{equation*}
    the minimum is negative if $c$ is large enough. Since
    \begin{equation*}
	\psi(\mu)\geq-\lambda_1>0\quad\text{for any }\mu\geq 0\quad\text{if }c=0,
    \end{equation*}
    the minimum is positive if $c$ is close enough to $0$. Moreover, the minimum is either
    located at $\mu=0$, in which case its value is $-\lambda_1'>0$, or it is located at some $\mu^\star>0$,
    in which case its value is decreasing with respect to $c$. By continuity, the minimum is, as a function of
    $c$, positive and constant in some interval $[0,c^\dagger)$, with $c^\dagger\geq 0$, and decreasing in 
    $[c^\dagger,+\infty)$, with a positive value at $c=c^\dagger$ and with limit $-\infty$ as $c\to+\infty$.

    By continuity, strict convexity, strict monotonicity, there exists a threshold $\hat{c}>0$
    such that the equation $\psi(\mu)=0$ admits therefore:
    \begin{itemize}
	\item no solution if $c\in(0,\hat{c})$;
	\item exactly one solution $\mu^\star$ if $c=\hat{c}$;
	\item exactly two isolated solutions $\mu^\star_1<\mu^\star_2$ if $c>\hat{c}$.
    \end{itemize}
    In view of the sign of $\psi(0)$, these solutions, if any, are positive.
    Hence the image of $\mu\in(0,+\infty)\mapsto-\lambda_{1,-\mu e}/\mu$
    contains $[\hat{c},+\infty)$ and does not contain $[0,\hat{c})$; in other words,
    it is exactly $[\hat{c},+\infty)$, whence $\hat{c}=c^{\mu^\star}_e=\min_{\mu>0}c^\mu_e$.
\end{proof}

\begin{lem}\label{lem:minimal_e}
    \red{The infimum $c^{\upFG}_e\in[0,+\infty)$ of the set $\{ c^\star_{e'}\ |\ e'\in\mathbb{S}^{n-1}, e\cdot e'>0 \}$
    is a minimum.

    Consequently, $c^{\upFG}_e>0$.}
\end{lem}

\begin{proof}
    It is sufficient to show that $(c^\star_{e'})_{e'\in\mathbb{S}^{n-1}}$ is 
    uniformly positive and globally bounded, so that the quantity $c^\star_{e'}/e\cdot e'$
    is positive if $e\cdot e'>0$ and tends to $+\infty$ if in addition $e\cdot e'\to 0$.

    To this end, it suffices to observe that $e\in\mathbb{S}^{n-1}\mapsto c^\star_e$ is continuous, 
    defined on a compact set, and pointwise positive (due to $c^\star_e\geq-\lambda_1/\mu^\star>0$).
\end{proof}

The main idea of the forthcoming proof is to compare, again, \eqref{sys:KPP} to cooperative systems for
which the Freidlin--G\"{a}rtner formula is easier to establish. 

Just as in the proof of Theorem \ref{thm:hair_trigger_effect}, we fix the solution $\veu$ and
apply Corollary \ref{cor:comparison_of_components_in_real_time}. This makes it possible to compare from below 
and starting from $t=1$ the reaction term
$\veL\veu-(\veC\veu)\circ\veu$ to a cooperative reaction term $\veL\veu-D\veu^{\circ(1+p)}$, $D>0$. 
By the comparison principle \red{of Proposition \ref{prop:comparison_principle}} applied to this new cooperative
reaction term, any solution of $\cbQ\underline{\veu}=-D\underline{\veu}^{\circ(1+p)}$ with
$\vez\leq\underline{\veu}(1,\cdot)\leq\veu(1,\cdot)$ will satisfy $\underline{\veu}\leq\veu$ globally in space at any time $t\geq 1$.

Similarly, by nonnegativity and global boundedness of $\veC$ (\textit{cf.} \ref{ass:KPP}, \ref{ass:smooth_periodic}), 
we can compare from above and starting from $t=0$
the reaction term $\veL\veu-(\veC\veu)\circ\veu$ to another cooperative reaction term $\veL\veu-D'\veu^{\circ 2}$,
$D'>0$. By the comparison principle, any solution of $\cbQ\overline{\veu}=-D'\overline{\veu}^{\circ 2}$ with 
$\overline{\veu}(0,\cdot)\geq\veu(0,\cdot)$ will satisfy $\overline{\veu}\geq\veu$ globally in space at any time $t\geq 0$.

Consequently, \eqref{sys:KPP} can be compared from above and from below, in times large enough, to cooperative systems of the form 
\begin{equation*}
    \cbQ\vev=-g\vev^{\circ(1+q)}\quad\text{for some }g>0,\ q>0
\end{equation*}
\red{and with nonnegative nonzero compactly supported initial conditions $\vev_{\upini}$ (using the compact support of
$\veu_{\upini}=\veu(0,\cdot)$ and the pointwise positivity of $\veu(1,\cdot)$)}.
Since these semilinear systems share the same linear part $\cbQ\vev$, and since their nonlinear part $-g\vev^{\circ(1+q)}$
has a constant negative sign,
it is expected that they all satisfy the Freidlin--G\"{a}rtner formula, and therefore that they all have the same
spreading speed $c^{\upFG}_e$ (independent of $g$). Proving this claim will end our proof. To do so, we will use 
recent results from Du--Li--Shen \cite{Du_Li_Shen_2022} as well as a delicate construction from Berestycki--Hamel--Nadin 
\cite{Berestycki_Hamel_Nadin}.

\begin{lem}
    Let $q>0$ and $\vect{g}\in\caC_{\upp}^{\delta/2,\delta}(\R\times\R^n,(\vez,\vei))$.

Then any solution $\vev$ of $\cbQ\vev=-\diag(\vect{g})\vev^{\circ(1+q)}$ in $(0,+\infty)\times\R^n$
    with nonnegative nonzero compactly supported initial condition
    $\vev_{\upini}$ spreads in the direction $e$ at speed $c^{\upFG}_e$, namely 
    \begin{equation*}
	\liminf_{t\to+\infty}\min_{i\in[N]}\inf_{|x|\leq R}v_i(t,x+cte)>0\quad\text{for all }R>0\text{ and }c\in(0,c^{\upFG}_e),
    \end{equation*}
    \begin{equation*}
	\lim_{t\to+\infty}\max_{i\in[N]}\sup_{|x|\leq R}v_i(t,x+cte)=0\quad\text{for all }R>0\text{ and }c>c^{\upFG}_e.
    \end{equation*}
\end{lem}

\begin{proof}
    Let 
    \begin{equation*}
	\kappa\geq\max\left(\left(\frac{-\lambda_1'}{\displaystyle\min_{(i,t,x)\in[N]\times\R\times\R^n}g_i(t,x)u_{0,i}(t,x)^q}\right)^{1/q},\max_{(i,x)\in[N]\times\R^n}v_{\upini,i}(x)\right)>0.
    \end{equation*}
    Then the function $\kappa\veu_0$ satisfies straightforwardly
    \begin{equation*}
	\begin{cases}
	    \cbQ(\kappa\veu_0)+\diag(\vect{g})(\kappa\veu_0)^{\circ(1+q)}\geq\vez, \\
	    \kappa\veu_0(0,\cdot)\geq\vev_{\upini},
	\end{cases}
    \end{equation*}
    and, by virtue of the comparison principle of Proposition \ref{prop:comparison_principle},
    $\kappa\veu_0\geq\vev$ in $[0,+\infty)\times\R^n$. 
    Hence $\vev$ is globally bounded, in a way that only depends on 
    $\max_{(i,x)\in[N]\times\R^n}v_{\upini,i}(x)$. Furthermore, a similar application
    of the comparison principle and a classical minimization of the parameter $\kappa$ show that
    space-time periodic solutions of $\cbQ\vev=-\diag(\vect{g})\vev^{\circ(1+q)}$ are 
    \textit{a priori} uniformly globally bounded, in a way reminiscent of Corollary 
    \ref{cor:global_bound_entire}.
    Then, by arguments very similar to those proving Theorem \ref{thm:existence_persistent_entire_solution}, there exists
    a nonnegative nonzero space-time periodic entire solution $\vev^\star$. 
    Similarly, the proof of Theorem \ref{thm:hair_trigger_effect} can be readily adapted to show that all solutions $\vev$
    of the Cauchy problem persist locally uniformly, namely they satisfy \eqref{eq:Cauchy_persistence}. In particular,
    $\vev^\star\gg\vez$ in $\R\times\R^n$.

    Now, in order to be in a position to apply \cite[Theorem 2.1]{Du_Li_Shen_2022}, 
    we transform the unknown $\vev$ in such a way that the structure of the problem is preserved but 
    the space-time periodic entire solution $\vev^\star$ is replaced by
    the space-time homogeneous steady state $\veo$. To do so, we set $\widetilde{\vev}$ so that $\vev=\widetilde{\vev}\circ\vev^\star$. 
    \red{Then, for each $i\in[N]$, standard calculus and the symmetry of $A_i$ (\textit{cf.} \ref{ass:smooth_periodic}) lead to
    \begin{equation*}
	\nabla\cdot(A_i\nabla v_i) 
	= \widetilde{v}_i\nabla\cdot(A_i\nabla v^\star_i) + v^\star_i\nabla\cdot(A_i\nabla\widetilde{v}_i)
	+2A_i\nabla v^\star_i\cdot\nabla \widetilde{v}_i
    \end{equation*}
    so that
    \begin{equation*}
	\caP_i v_i = \widetilde{v}_i \caP_i v^\star_i + v^\star_i \caP_i\widetilde{v}_i -2A_i \nabla v^\star_i\cdot \nabla \widetilde{v}_i.
    \end{equation*}
    Since $v^\star_i$ is positive pointwise, the terms can be rearranged in the following way:
    \begin{equation*}
	\caP_i\widetilde{v}_i-2\frac{A_i\nabla v^\star_i}{v^\star_i}\cdot\nabla\widetilde{v}_i = -\frac{\caP_i v^\star_i}{v^\star_i}\widetilde{v}_i+\frac{1}{v^\star_i}\caP_i v_i.
    \end{equation*}
    By using the equalities 
    \begin{equation*}
	\dcbP\vev^\star=\veL\vev^\star-\diag(\vect{g})(\vev^\star)^{\circ(1+q)},
    \end{equation*}
    \begin{equation*}
	\dcbP\vev=\veL(\widetilde{\vev}\circ\vev^\star)-\diag(\vect{g})(\widetilde{\vev}\circ\vev^\star)^{\circ(1+q)},
    \end{equation*}
    we deduce that} each $\widetilde{v}_i$, $i\in[N]$, satisfies:
    \begin{equation*}
	\caP_i\widetilde{v}_i-2\frac{A_i\nabla v^\star_i}{v^\star_i}\cdot\nabla\widetilde{v}_i = -\left( \sum_{j\in[N]}l_{i,j}\frac{v^\star_j}{v^\star_i}-g_i (v^\star_i)^q \right)\widetilde{v}_i
	+\sum_{j\in[N]}l_{i,j}\frac{v^\star_j}{v^\star_i}\widetilde{v}_j-g_i (v^\star_i)^q(\widetilde{v}_i)^{1+q}
    \end{equation*}
    By setting 
    \begin{equation*}
	\widetilde{q}_i=q_i-2\frac{A_i\nabla v^\star_i}{v^\star_i},
    \end{equation*}
    \begin{equation*}
	\widetilde{g}_i=g_i (v^\star_i)^q,
    \end{equation*}
    \begin{equation*}
	\widetilde{\veL}=\left( l_{i,j}\frac{v^\star_j}{v^\star_i} \right)_{(i,j)\in[N]^2}+\diag\left( g_i (v^\star_i)^q-\sum_{j\in[N]}l_{i,j}\frac{v^\star_j}{v^\star_i} \right)_{i\in[N]},
    \end{equation*}
    \begin{equation*}
	\widetilde{\cbQ}=\diag(\partial_t-\nabla\cdot(A_i\nabla)+\widetilde{q}_i)-\widetilde{\veL},
    \end{equation*}
    we obtain a new system
    \begin{equation*}
	\widetilde{\cbQ}\widetilde{\vev}=-\diag(\widetilde{\vect{g}})\widetilde{\vev}^{\circ(1+q)}.
    \end{equation*}
    This new system has the same structure as the original system $\cbQ\vev=-\diag(\vect{g})\vev^{\circ(1+q)}$
    and admits $\veo$ as space-time homogeneous solution by construction. 
    In particular, let us emphasize that the
    matrix $\widetilde{\veL}$ satisfies the structural assumptions \ref{ass:cooperative}, \ref{ass:irreducible}.

    Let us verify that $\lambda_1(\widetilde{\cbQ})<0$. Assume by contradiction $\lambda_1(\widetilde{\cbQ})\geq 0$.
    Then the proof of Theorem \ref{thm:extinction_small_solutions} can be readily adapted to construct solutions $\widetilde{\vev}$
    that vanish locally uniformly. But this, in turn, implies the existence of solutions $\vev$ of the original system 
    $\cbQ\vev=-\diag(\vect{g})\vev^{\circ(1+q)}$ that vanish locally uniformly, contradicting the locally uniform persistence 
    of all solutions.

    Let us verify now that $\veo$ is globally attractive for solutions $\widetilde{\vev}$ whose initial values are uniformly positive, 
    space-periodic and valued in $[\vez,\veo]$. For any $T\geq 0$, let
    \begin{equation*}
	M_T=\min_{i\in[N]}\min_{x\in[0,L]}\widetilde{v}_i(T,x),
    \end{equation*}
    so that
    \begin{equation*}
	M_T\veo\leq\widetilde{\vev}(T,\cdot)\leq\veo\quad\text{in }\R^n. 
    \end{equation*}
    \red{Then $M_T\veo$ is a sub-solution of the cooperative semilinear system starting from $t=T$; indeed, 
    \begin{align*}
	\widetilde{\cbQ}(M_T\veo)+\diag(\widetilde{\vect{g}})(M_T\veo)^{\circ(1+q)} & = M_T\left(-\widetilde{\veL}\veo+M_T^q\diag(\widetilde{\vect{g}})\veo\right) \\
	& = M_T\left(-\diag(\widetilde{\vect{g}})\veo+M_T^q\diag(\widetilde{\vect{g}})\veo\right) \\
	& = M_T(M_T^q-1)\widetilde{\vect{g}} \\
	& \leq\vez
    \end{align*}
    where we have used the facts that $\veo$ is a solution of the system and that, necessarily, $M_T\leq 1$.}
    Similarly, $\veo$ is a global super-solution. Hence it suffices to prove $M_T\to 1$ as $T\to+\infty$. 
    By applying the strong comparison principle in a way similar to the proof of Theorem \ref{thm:extinction} (case $\lambda_1'=0$), 
    we deduce that $T\mapsto M_T$ is increasing. Hence it converges to a limit $M_\infty\in(0,1]$. If $M_\infty<1$, then by 
    a limiting argument again similar to that of the proof of Theorem \ref{thm:extinction}, we find a new space-periodic
    entire solution $\widetilde{\vev}_\infty$ valued in $[M_\infty\veo,\veo]$, and then by comparison with the sub-solution
    $M_\infty\veo$, a contradiction arises, just as in the proof of Theorem \ref{thm:extinction}.
    Hence $M_\infty=1$ and $\veo$ is indeed globally attractive for uniformly positive, space-periodic solutions in $[\vez,\veo]$.

    Therefore the transformed system $\widetilde{\cbQ}\widetilde{\vev}=-\diag(\widetilde{\vect{g}})\widetilde{\vev}^{\circ(1+q)}$ 
    satisfies the assumptions 
    \textbf{(A1)}--\textbf{(A6)} of \cite[Theorem 2.1]{Du_Li_Shen_2022}. Consequently, its solutions with planar Heaviside-like
    initial conditions in some direction $e'\in\mathbb{S}^{n-1}$, namely initial conditions in
    \begin{equation*}
	\vect{H}_{e'} = \left\{\widetilde{\vev}_{\upini}\in L^\infty(\R^n,[\vez,\veo])\ |\ \liminf_{x\cdot e'\to-\infty}\widetilde{\vev}_{\upini}\gg\vez,\ 
	\exists B\in\R\quad (\widetilde{\vev}_{\upini})_{|\{x\cdot e'\geq B\}}=\vez \right\},
    \end{equation*}
    spread at least at speed $c_{\inf}(e')$ and at most at speed $c_{\sup}(e')$, where
    \begin{equation*}
	c_{\inf}(e')=\sup\left\{c\geq 0\ |\ \widetilde{\vev}_{\upini}\in\vect{H}_{e'}\implies \lim_{t\to+\infty}\inf_{x\cdot e'\leq ct} \widetilde{\vev}(t,x)=\veo\right\},
    \end{equation*}
    \begin{equation*}
	c_{\sup}(e')=\inf\left\{c\geq 0\ |\ \widetilde{\vev}_{\upini}\in\vect{H}_{e'}\implies \lim_{t\to+\infty}\sup_{x\cdot e'\geq ct} \widetilde{\vev}(t,x)=\vez\right\}.
    \end{equation*}

    Let us prove now that
    \begin{equation*}
	c_{\inf}(e')=c_{\sup}(e')=c^\star_{e'}\quad\text{for any }e'\in\mathbb{S}^{n-1}.
    \end{equation*}
    More precisely, since $c_{\inf}(e')\leq c_{\sup}(e')$ is clear,
    we are going to prove $c_{\sup}(e')\leq c^\star_{e'}$ on one hand
    and $c_{\inf}(e')\geq c^\star_{e'}$ on the other hand.
    
    The inequality $c_{\sup}(e')\leq c^\star_{e'}$ follows from the super-solution
    \begin{equation*}
	\overline{\vev}:(t,x)\mapsto M\upe^{-\lambda_{1,z^\star}t}\upe^{z^\star\cdot x}\veu_{z^\star}(t,x),
    \end{equation*}
    where $z^\star=-\mu^\star_{e'}e'$ and with $M>0$ so large that $\overline{\vev}(0,\cdot)\geq\vev_{\upini}$. Here $\mu^\star_{e'}$
    denotes obviously the unique $\mu^\star$ associated with direction $e'$ given by Lemma \ref{lem:minimal_c}.
    In view of the invertible change of unknown $\vev=\widetilde{\vev}\circ\vev^\star$ and of the above calculations, 
    the inequality $\widetilde{\overline{\vev}}\geq\widetilde{\vev}$ is clear globally in $[0,+\infty)\times\R^n$.
    Moreover, by definition of $c^\star_{e'}$,
    \begin{equation*}
	\exp(-\lambda_{1,z^\star}t)\exp(-\mu^\star_{e'} x\cdot e')=\exp(-\mu^\star_{e'}\left( x\cdot e'-c^\star_{e'} t\right)),
    \end{equation*}
    whence the super-solution spreads exactly at speed $c^\star_{e'}$ and, in view of its exponential decay at $x\cdot e'=+\infty$,
    this implies $c_{\sup}(e')\leq c^\star_{e'}$.

    The inequality $c_{\inf}(e')\geq c^\star_{e'}$ follows similarly from the construction of a Heaviside-like sub-solution
    that spreads at some speed $c<c^\star_{e'}$. Since this construction is quite long -- density of directions $e'$ meeting the
    spatial periodicity network, approximation in straight cylinders in direction $e'$, existence of principal eigenfunctions 
    in such cylinders by complex analysis arguments, and finally continuation of the spreading speed estimate for directions 
    $e'$ that do not meet the spatial periodicity network -- but analogous to the scalar construction in 
    \cite[Sections 4.2, 4.3]{Berestycki_Hamel_Nadin}, it is not detailed here. Let us just mention two specificities of 
    the vector setting worthy of attention: 
    \begin{itemize}
	\item The spectral approximation in cylinders of increasing radius, stated in \cite[Proposition 4.4]{Berestycki_Hamel_Nadin},
	    is proved thanks to ratios of scalar quantities whose vector generalization seems \textit{a priori} unclear. 
	    It can however be noticed that this proof is analogous to that of the spectral approximation in balls of increasing
	    radius, that we already generalized to the vector setting in \cite[Propositions 3.2, 3.9]{Girardin_Mazari_2022}. 
	    Adapting the proof of \cite[Propositions 3.2, 3.9]{Girardin_Mazari_2022} to cylinders is straightforward.
	\item Different components of the oscillating eigenfunction might vanish at different locations, so that by taking the 
	    positive part we do not obtain in general the solution of a linear Dirichlet problem as in
	    \cite[Section 4.3]{Berestycki_Hamel_Nadin}. Instead, by \red{Proposition \ref{prop:generalized_comparison_principle},
	    we obtain a generalized sub-solution of this linear Dirichlet problem. Since this part of the proof consists precisely
	    in constructing a sub-solution, this is sufficient and the proof can be kept unchanged.}
    \end{itemize}
    
    Consequently,
    \begin{equation*}
	c_{\inf}(e')=c_{\sup}(e')=c^\star_{e'}\quad\text{for all }e'\in\mathbb{S}^{n-1}.
    \end{equation*}

    We are now in a position to apply \cite[Theorem 2.3]{Du_Li_Shen_2022} and obtain the Freidlin--G\"{a}rtner formula
    for the spreading speed $c^{\upFG}_e$ in the direction $e$ of solutions of the cooperative system 
    $\widetilde{\cbQ}\widetilde{\vev}=-\diag(\widetilde{\vect{g}})\widetilde{\vev}^{\circ(1+q)}$ with compactly supported 
    initial conditions. 
    Going back to
    the original unknown $\vev$, we deduce as claimed the Freidlin--G\"{a}rtner formula for the spreading speed in the
    direction $e$ of solutions of the cooperative system $\cbQ\vev=-\diag(\vect{g})\vev^{\circ(1+q)}$ with compactly
    supported initial conditions.
\end{proof}

\begin{section}{Acknowledgements}
    L. G. acknowledges support from the ANR via the project Indyana under grant agreement ANR-21-CE40-0008
    and via the project Reach under grant agreement ANR-23-CE40-0023-01. 
    He thanks Christopher Henderson for fruitful discussions on Harnack inequalities and Gaussian estimates.
\end{section}

\bibliographystyle{plain}
\bibliography{ref}

\end{document}